\documentclass[11pt]{amsart}
\usepackage{amssymb}
\usepackage{amsmath}
\usepackage{amsthm}
\usepackage{amsfonts}
\usepackage{amsthm}
\usepackage{pgfplots}
\usepackage[alphabetic]{amsrefs}

\usepackage{palatino}
\usepackage{geometry}
\usepackage[hyperfootnotes=false, colorlinks, linkcolor={blue},
citecolor={magenta}, filecolor={blue}, urlcolor={blue}]{hyperref}

\newtheorem{thm}{Theorem}[section]
\newtheorem{prop}[thm]{Proposition}
\newtheorem{cor}[thm]{Corollary}
\newtheorem{lem}[thm]{Lemma}
\newtheorem{ex}[thm]{Example}
\newtheorem{rem}[thm]{Remark}

\newtheorem{conj}[thm]{Conjecture}
\newtheorem{con}[thm]{Construction}

\numberwithin{equation}{section}
\numberwithin{table}{section}

\newcommand{\C}{\mathbb{C}}

\newcommand{\tr}{\mathrm{tr}}

\newcommand{\cbc}{\mathrm{cbc}}
\newcommand{\barth}{\mathcal{Z}}
\newcommand{\Spec}{\mathrm{Spec}}

\newcommand{\bmx}{\begin{bmatrix}}
\newcommand{\emx}{\end{bmatrix}}
\renewcommand{\phi}{\varphi}
\newcommand{\pr}{\dagger}

\newcommand{\DZZ}{DZ$\bar{\textrm{Z}}$}
\newcommand{\DZb}{D$\barth$}
\newcommand{\DZZb}{D$\barth \bar \barth$}

\usepackage{tikz}
\usepackage{tkz-graph}
\usepackage{tkz-berge}
\usetikzlibrary{arrows,shapes}
\usepackage{enumitem}

\begin{document}

\title{Distinguishing graphs with zeta functions and generalized spectra}
\date{\today}
\author[Christina Durfee]{Christina Durfee}
\address[Christina Durfee]{Department of Mathematics, University of Oklahoma, Norman, OK 73019}
\email{cdurfee@math.ou.edu}

\author[Kimball Martin]{Kimball Martin$^\ast$}
\address[Kimball Martin]{Department of Mathematics, University of Oklahoma, Norman, OK 73019}
\email{kmartin@math.ou.edu}
\thanks{${}^\ast$Partially supported by a Simons Collaboration Grant.}


\begin{abstract}
Conjecturally, almost all graphs are determined by their spectra.  This problem has also been studied for variants such as the spectra
of the Laplacian and signless Laplacian.  Here we consider the problem of determining graphs with
Ihara and Bartholdi zeta functions, which are also computable in polynomial time.
These zeta functions are geometrically motivated, but can be viewed as certain generalizations of
characteristic polynomials.  After discussing some graph properties determined by zeta functions, we show that large classes of cospectral
graphs can be distinguished with zeta functions and enumerate graphs distinguished by zeta functions on $\le 11$ vertices.
This leads us to conjecture that almost all graphs which are not determined by their spectrum are determined by zeta functions.

Along the way, we make some observations about the usual types of spectra and disprove a conjecture of Setyadi and Storm
about Ihara zeta functions determining degree sequences.
\end{abstract}

\maketitle

\section{Introduction}

A fundamental problem in spectral graph theory is: when can we distinguish (unlabeled simple) graphs by their spectra?
If the answer were always, as was once conjectured, that would mean we could solve the graph isomorphism problem
in polynomial time.  However, many pairs of non-isomorphic cospectral graphs have
since been found, and various constructions of
cospectral pairs are known.  Nevertheless, Haemers conjectured that almost all graphs are determined by their spectra (DS),
i.e., the fraction of graphs of order $n$ which are DS goes to 1 as $n \to \infty$.
In fact, from the numerical data for $n \le 11$ in Haemers--Spence \cite{HS}, it appears that more graphs
are determined by their Laplacian spectra ($L$-DS), and even more by their signless Laplacian spectra ($|L|$-DS).
We refer to surveys by van Dam and Haemers for more details \cite{vDH1}, \cite{vDH2}.

This question of which graphs are DS can be thought of geometrically.  Knowing the adjacency spectrum of $G$ is equivalent to knowing
what one might call the {\em walk length spectrum} of $G$---the set $\{ (\ell, w_G(\ell)) : \ell \ge 0 \}$, where $w_G(\ell)$ is the number of
closed walks of length $\ell$ on $G$---together with the order $n$ of $G$.  Hence the above question can be stated:
when do the order and walk length spectrum determine $G$?

From the point of view of Riemannian geometry and covering space theory, it is more natural to look at {\em geodesics} on $G$ rather than arbitrary walks.  Roughly, geodesics
are paths with no backtracking and they correspond to lines in the universal cover of $G$ (see
Section \ref{sec:ihara}).  The {\em geodesic length spectrum} of $G$ is then the set
$\{ (\ell, a_G(\ell)) : \ell \ge 0 \}$ of numbers of (primitive)
closed geodesics of length $\ell$ for all $\ell$.
For graphs, the Ihara zeta function $Z_G(t)$ of $G$ encodes the geodesic length spectrum, and knowing one is equivalent to knowing the other.

This is entirely analogous to the Selberg zeta function encoding the (geodesic) length spectrum for Riemann surfaces.  For compact Riemann surfaces, Huber's theorem says that
knowing the Selberg zeta function, i.e., the length spectrum, is equivalent to knowing the spectrum of the Laplacian (see, e.g., \cite{buser}).
Similarly, for connected regular graphs, knowing the Ihara zeta function
is equivalent to knowing the (adjacency or Laplacian) spectrum (see Section \ref{sec:ihara-prop}).
However, this is not the case for irregular graphs (cf.\ Table \ref{tab:md2}).

Here, we suggest that the Ihara zeta function---which is computable in polynomial time---provides a more effective way to differentiate (irregular) graphs than the usual
spectra studied (Conjecture \ref{conj:main}).
One heuristic for why this should be the case is that geodesics capture much of the geometry of the graph better than arbitrary paths.
Another is that the Ihara zeta function typically encodes more information---at least for md2 graphs (graphs with no vertices of degree $\le 1$),
the reciprocal of the Ihara zeta function is a polynomial whose degree is twice the number of edges, and thus typically has more coefficients than
the characteristic polynomial.

The obvious drawback
of the Ihara zeta function is that it cannot detect leaves or isolated nodes.  For this reason (at least in part), most studies of the Ihara zeta function
restrict to md2 graphs.  Most studies also restrict to connected graphs, which is convenient to discuss
covering space theory, but this is less crucial for the problem of distinguishing graphs.  (As with characteristic polynomials, the zeta function of a disconnected
graph is the product of the zeta functions of its components.)  Instead of making such restrictions here, we consider four methods of addressing this defect
to distinguish graphs (undefined notation explained subsequently):

\begin{enumerate}
\item[(M1)] Use the order and the Ihara zeta $Z_{G^*}(t)$ function of the cone $G^*$ of $G$.  Equivalently use the order and the characteristic polynomial
$\phi_{T^*}(\lambda) = \det(\lambda I - T^*)$.

\item[(M2)] Use the order, size and the Ihara zeta functions $Z_G(t)$ and $Z_{\bar G}(t)$
of $G$ and its complement $\bar G$.
Equivalently, use the order and the characteristic polynomials
$\phi_{T}(\lambda) = \det(\lambda I - T)$ and $\phi_{\bar T}(\lambda) = \det(\lambda I - \bar T)$.

\item[(M3)] Use the order and the Bartholdi zeta function $\barth_G(t,u)$ of $G$.
Equivalently, use the
generalized characteristic polynomial $\phi_{AD}(\lambda, x) = \det(\lambda I - A + xD)$.

\item[(M4)] Use the order and the Bartholdi zeta functions  $\barth_G(t,u)$ and
$\barth_{\bar G}(t,u)$ of $G$ and $\bar G$.
Equivalently, use the generalized characteristic polynomial $\phi_{ADJ}(\lambda, x, y) = \det(\lambda I - A + xD + yJ)$.
\end{enumerate}

Here $I$ and $J$ denote the identity and all-ones matrices
of the appropriate sizes, and $A$ and $D$ denote the adjacency and degree matrices for a fixed ordering of vertices on $G$.
Further, $T$ (resp.\ $T^*$, $\bar T$) denotes Hashimoto's oriented edge matrix (see Section \ref{sec:ihara}) of $G$ (resp.\ $G^*$, $\bar G$)---whose spectrum contains the same
information as $Z_G$ together with the size of $G$ (cf.\ \eqref{eq:hashimoto-det}).
The two-variable Bartholdi zeta function $\barth_G(t,u)$ is a generalization of the
one-variable Ihara function
$Z_G(t): \C \to \C$ which counts ``closed geodesics with $r$ backtracks''---see Section \ref{sec:bartholdi}, in particular Equation \eqref{eq:phiAD} for the relation with $\phi_{AD}$.

Note that knowing $\phi_{AD}$ implies knowing the spectra of $A$, the Laplacian $L$, and
the signless Laplacian $|L|$.
Similarly knowing $\phi_{ADJ}$ implies knowing all of these spectra for both $G$ and $\bar G$.

We will explain the precise motivation behind these choices as we consider each in turn in the body of the paper.
The basic ideas are that the Ihara zeta function cannot detect ``dangling edges,'' i.e., edges to a degree 1 node, but (i) taking cones
or complements turns degree 1 nodes into higher degree nodes and (ii)  the Bartholdi zeta function does detect dangling edges.
In analogy with the DS terminology, we say that a graph $G$ is DZ* (resp.\ \DZZ, \DZb, \DZZb) if it is uniquely determined by
method (M1) (resp.\ (M2), (M3), (M4)).

We will see in Section \ref{sec:M4} that (M4) is provably stronger than each of (M1), (M2) and (M3), i.e.,
being DZ* (resp.\ \DZZ, \DZb) implies \DZZb.  However there are no obvious implications
among (M1), (M2) and (M3) (but there are some indications that (M2) may be stronger than (M1)---see the end of Section \ref{sec5}).  Nevertheless, a surprising outcome of our calculations in
Table \ref{tab:main} is that for
graphs on $n \le 11$ vertices, methods (M1), (M2) and (M4) are always equivalent (whereas (M3) has slightly less discriminating power).
This suggests one gains little by using the Bartholdi zeta function, and one is in practice justified in
just considering the Ihara zeta function, which is both conceptually and computationally simpler.
However, all of these methods run in polynomial time.

\medskip
Now we outline the contents of the paper.

In Section \ref{sec:ihara}, we first recall basic facts about the Ihara zeta function,
and discuss some graph properties it determines.  Then we consider some basic properties determined by
methods (M1) and (M2).  In particular, these zeta invariants force strong restrictions on the degree sequence of a graph.  Setyadi--Storm \cite{SS}, in their enumeration of pairs of connected
md2 graphs with the same Ihara zeta function on $n \le 11$ vertices,
found that the Ihara zeta function of a connected md2 graph determines the degree sequence for $n \le 11$, and conjectured
this holds for all $n$, but we give a counterexample to this conjecture on 12 vertices
(Example \ref{ex:crabsquid}).
Nevertheless, we show that knowing the zeta functions of sufficiently many cones
of $G$ algorithmically determines the degree sequence (Lemma \ref{lem:ds}).

In Section \ref{sec:bartholdi}, we recall basic facts about the Bartholdi zeta function and
discuss the relative strengths of methods (M1)--(M4).  To show that (M4) is stronger than (M1),
we show that two graphs have the same spectra with respect to $A+xD$ and $\bar A+ x\bar D$ for fixed $x$,
 if and only if the same is true of their joins with another
graph (Theorem \ref{thm:cone-comp}).  In particular, two graphs have the same $A$- and $\bar A$- (or
$|L|$- and $\overline{|L|}$-) spectra if and only if the same is true for their cones
(Corollary \ref{cor:cone-comp}; cf.\ Table \ref{tab:cone-comp} for related data).

In Section \ref{sec:constructions}, we show that methods (M1)--(M4)
distinguish a large class of cospectral and Laplacian cospectral
pairs coming from well known constructions:  GM switching, coalescence and join.
On the other hand,
these methods will not distinguish graphs arising from the more restrictive
GM* switching.  Section \ref{sec:new-con} presents a new construction of pairs of graphs
which cannot be distinguished by these methods.  This construction is interesting because
it constructs graphs $G_1$, $G_2$ such that the generalized adjacency matrices $A_1 + xD_1$
and $A_2 + xD_2$ are miraculously conjugate for all $x$ but not ``uniformly conjugate''
(as happens in GM* switching)
i.e., there is no invertible matrix $P$ such that $P(A_2+xD_2)P^{-1} = A_1 + xD_1$
for all $x$.

In Section \ref{sec5}, we enumerate all graphs on $\le 11$ vertices which are not
DZ* (resp.\ \DZZ, \DZb, \DZZb)  in Table \ref{tab:main}.  We state Conjecture \ref{conj:main},
which asserts almost all graphs are DZ* (resp.\ \DZZ, \DZb, \DZZb), and further
that almost all non-DS graphs are DZ* (resp.\ \DZZ, \DZb, \DZZb).
Table \ref{tab:md2} indicates that there is no essential
difference in just using the Ihara zeta function $Z_G$ compared to methods (M1)--(M4)
when restricting to md2 graphs.  This suggests Conjecture \ref{conj:md2}, which is the
analogue of Conjecture \ref{conj:main} for determining md2 graphs $G$ using only $Z_G$.
We also compare the effectiveness of combining different kinds of
spectra for $n \le 11$ in Table \ref{tab:combine}.  In particular, this suggests that
using two of the usual spectra, such as $A$ and $L$ or $A$ and $|L|$, is much more effective
at distinguishing graphs than any single one.

Our calculations were done using Sage \cite{sage}, including nauty \cite{nauty},
and standard Unix tools.  In practice, we found that to
check if two graphs have the same zeta function, it almost always sufficed
to compute the two numbers $\det |L|$ and $\det(4D+2A-3I)$, which are essentially (the
residue of) $Z_G(1)$
and $Z_G(-2)$.
For convenience of the interested reader, when mentioning particular examples of graphs, we
include (non-canonical) {\verb+graph6+} strings to specify the graphs, which can be used
to easily reconstruct the graphs in Sage.

As a final remark, we note that there are more general notions of zeta functions of graphs, such
as path and edge zeta functions (see \cite{ST} or \cite{terras}).  We do not consider these here.

\medskip
We thank C.\ Storm for discussions about \cite{SS}, as well as the referee for helpful comments.

\section{The Ihara zeta function} \label{sec:ihara}

We begin by defining closed geodesics and the Ihara zeta function.
This notion of a geodesic on a graph corresponds to a bi-infinite simple path in the universal
cover, but we do not explain this here---see, e.g., \cite{ST} or \cite{terras}.
Ihara zeta functions are a special case of more general zeta functions of multigraphs
considered by Hashimoto \cite{hashimoto}, generalizing the zeta functions originally defined by Ihara
\cite{ihara}.

Fix a finite (simple) graph $G = (V, E)$.  Let $n = |V|$ and $m = |E|$.  We will denote walks by
sequences of adjacent oriented (or directed) edges $e_i$.  For an oriented edge $e_i = (u, v)$, let
$e_i^{-1} = (v,u)$ be the edge with reversed orientation.
Suppose $\gamma = (e_1, e_2, \ldots, e_\ell)$ denotes a closed
walk of length $\ell$ in $G$.  We say $\gamma$ is
a {\em closed geodesic} if $e_i \ne e_{i+1}^{-1}$ for $1 \le i < \ell$
and $e_1 \ne e_{\ell}^{-1}$.  The former condition is often expressed saying $\gamma$ has no
backtracking and the latter that $\gamma$ has no tails.  Write $k\gamma$ for the concatenation
of $k$ copies of $\gamma$.  We say $\gamma$ is {\em primitive} if $\gamma$ is not of
the form $k \delta$ for a closed geodesic $\delta$ and some $k \ge 2$.

Let $\sigma(\gamma) = (e_2, \ldots, e_\ell, e_1)$.
Then $\sigma(\gamma)$ is also a closed geodesic of length $\ell$, and it is primitive if $\gamma$ is.
Thus the cyclic group $\langle \sigma \rangle$ of order $\ell$  generated
by $\sigma$ acts on all (primitive) closed geodesics of length $\ell$.  The
$\langle \sigma \rangle$-orbits thus partition
the set of (primitive) closed geodesics of length $\ell$ into equivalence classes.
Let $a(\ell)=a_G(\ell)$ denote the number of
primitive closed geodesics of length $\ell$ up to equivalence.

Note that, as $G$ is simple, there are no closed geodesics of length $< 3$.  The closed geodesics of length 3, 4 or 5
are just the cycles of length 3, 4 or 5.  The closed geodesics of length 6
are just the cycles of length 6 together with the concatenations of 2 simple cycles
of length 3 starting at a fixed base point.
If there are two distinct cycles based at a vertex $v_0$, going around the
first cycle $k$ times and going around the second cycle once or more is a primitive closed geodesic.
Hence, for connected md2 graphs, the lengths of primitive closed geodesics are unbounded
unless $G$ is a circuit.

The (Ihara) zeta function of $G$ is
\begin{equation} \label{eq:zeta-def}
Z_G(t) = \prod_\gamma (1-t^{\ell(\gamma)})^{-1}
= \prod_{\ell > 2} (1-t^{\ell})^{-a(\ell)} = \exp \left( \sum_{\ell > 2} \sum_{k \ge 1} a(\ell)  \frac{t^{\ell k}}k \right),
\end{equation}
where, in the first product, $\gamma$ runs over a set of representatives for the equivalences classes of primitive
closed geodesics, and $\ell(\gamma)$ denotes the length of $\gamma$.
Since no closed geodesics will involve degree 0 or degree 1 nodes, $Z_G(t) = Z_{G^\pr}(t)$,
where $G^\pr$ is the ``pruned graph'' obtained by successively deleting degree 0 and degree 1
nodes until one is either left with an md2 graph or the null graph (the graph on 0 vertices).

Note if $G$ is a disjoint union of two subgraphs, $G = G_1 \sqcup G_2$, then
$a_G(\ell) = a_{G_1}(\ell) + a_{G_2}(\ell)$.  Hence the zeta function of a graph is the product
of the zeta functions of its connected components.

The zeta function is {\em a priori} an infinite product, but turns out to be
a rational function and thus is meromorphic on $\C$.
Namely, Hashimoto \cite{hashimoto}, \cite{hashimoto3} and Bass \cite{bass}
gave two determinant formulas for $Z_G$, which have been subsequently
retreated many times (e.g., \cite{ST}).
The Bass determinant formula \cite{bass} is
\begin{equation} \label{eq:bass-det}
 Z_G(t) =  (1-t^2)^{n-m} \det(I - tA + t^2(D-I))^{-1}.
\end{equation}
Note the right hand side is invariant under
adding nodes of degree 0 or 1 to $G$.

Let $\{ e_1, \ldots, e_{2m} \}$ denote the set of oriented edges of $G$.
The {\em oriented edge matrix} (with respect to this ordering of oriented edges) $T$
is the $2m \times 2m$ matrix whose $(i,j)$-entry is 1 if $e_i = (u, v)$, $e_j = (v, w)$ and
$u \ne w$; or 0 otherwise.
The Hashimoto determinant formula \cite{hashimoto}, \cite{hashimoto3}
(cf.\ \cite[Thm 3]{ST}) is
\begin{equation} \label{eq:hashimoto-det}
Z_G(t) = \det(I - tT)^{-1} = (t^{2m} \phi_T(t^{-1}))^{-1},
\end{equation}
where $\phi_T(\lambda) = \det(\lambda I - T)$ is the characteristic polynomial of $T$.
Again, one can check that $\det(I - tT)$
is invariant under adding degree 0 or 1 nodes to $G$.

Since $G^\pr \ne G$ in general, $Z_G$ does not determine $m$ (or $n$).
This means knowing $Z_G$ is not exactly the same as knowing the
spectrum of $T$, but it almost is:  $\phi_T$ determines both $Z_G$ and $m$,
and conversely.  Indeed, the degree of $\phi_T$ is $2m$, so one can recover
$Z_G$ from $\phi_T$.  The converse is obvious.  This observation
relates to method (M2).

\subsection{Properties determined by the Ihara zeta function} \label{sec:ihara-prop}

Here we summarize some elementary graph properties determined by the Ihara zeta function.
This question has previously been considered mainly for md2 and connected md2
multigraphs (e.g., \cite{czarneski} and \cite{cooper}).

First, examining the coefficients of $t^k$ in the logarithm of \eqref{eq:zeta-def}
shows $Z_G$ determines $a(\ell)$ for all $\ell$, hence
knowing $Z_G$ is the same as knowing the primitive geodesic length spectrum.  Note that the number of non-primitive
geodesics of a given length $l$ is the sum of the numbers of primitive geodesics of proper divisors of $l$.  Hence the primitive geodesic length spectrum determines the full geodesic length spectrum and vice versa.
Consequently $Z_G$ determines the number of cycles of length 3, length 4 and length 5
in $G$.  It is not true that $Z_G$ determines the number of cycles of length 6.

\begin{ex} The pair of
graphs {\verb+HheadXZ+} and {\verb+Hhf@eS|+} pictured below,
each with 9 vertices and 18 edges,
have the same zeta functions but a different number of
cycles of length 6 (46 and 50).

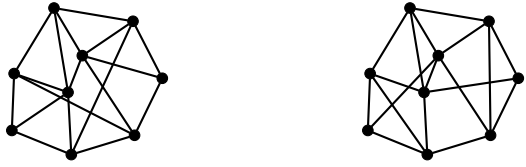
\begin{center}
\begin{tikzpicture}
\tikzset{VertexStyle/.style = {shape = circle, fill = black, inner sep = 0pt, outer sep = 0pt, minimum size = 4pt, draw} }
\useasboundingbox (0,0) rectangle (2.0cm,2.0cm);
\Vertex[L=\hbox{$.$},x=0.0303cm,y=1.0779cm]{v0}
\Vertex[L=\hbox{$.$},x=0.7905cm,y=0.0cm]{v2}
\Vertex[L=\hbox{$.$},x=1.6314cm,y=0.2568cm]{v1}
\Vertex[L=\hbox{$.$},x=1.6105cm,y=1.7737cm]{v3}
\Vertex[L=\hbox{$.$},x=0.5603cm,y=1.95cm]{v4}
\Vertex[L=\hbox{$.$},x=0.0cm,y=0.3213cm]{v5}
\Vertex[L=\hbox{$.$},x=2.0cm,y=1.0155cm]{v6}
\Vertex[L=\hbox{$.$},x=0.7469cm,y=0.8281cm]{v7}
\Vertex[L=\hbox{$.$},x=0.9369cm,y=1.3163cm]{v8}
\Edge[](v0)(v1)
\Edge[](v0)(v4)
\Edge[](v0)(v5)
\Edge[](v0)(v7)
\Edge[](v1)(v2)
\Edge[](v1)(v6)
\Edge[](v1)(v8)
\Edge[](v2)(v3)
\Edge[](v2)(v5)
\Edge[](v2)(v7)
\Edge[](v3)(v4)
\Edge[](v3)(v6)
\Edge[](v3)(v8)
\Edge[](v4)(v7)
\Edge[](v4)(v8)
\Edge[](v5)(v7)
\Edge[](v6)(v8)
\Edge[](v7)(v8)
\end{tikzpicture}
\hspace{1in}
\begin{tikzpicture}
\tikzset{VertexStyle/.style = {shape = circle, fill = black, inner sep = 0pt, outer sep = 0pt, minimum size = 4pt, draw} }
\useasboundingbox (0,0) rectangle (2.0cm,2.0cm);
\Vertex[L=\hbox{$.$},x=0.0303cm,y=1.0779cm]{v0}
\Vertex[L=\hbox{$.$},x=0.7905cm,y=0.0cm]{v1}
\Vertex[L=\hbox{$.$},x=1.6314cm,y=0.2568cm]{v2}
\Vertex[L=\hbox{$.$},x=1.6105cm,y=1.7737cm]{v3}
\Vertex[L=\hbox{$.$},x=0.5603cm,y=1.95cm]{v4}
\Vertex[L=\hbox{$.$},x=0.0cm,y=0.3213cm]{v5}
\Vertex[L=\hbox{$.$},x=2.0cm,y=1.0155cm]{v6}
\Vertex[L=\hbox{$.$},x=0.7469cm,y=0.8281cm]{v7}
\Vertex[L=\hbox{$.$},x=0.9369cm,y=1.3163cm]{v8}
\Edge[](v0)(v1)
\Edge[](v0)(v4)
\Edge[](v0)(v5)
\Edge[](v0)(v7)
\Edge[](v1)(v2)
\Edge[](v1)(v5)
\Edge[](v1)(v7)
\Edge[](v2)(v3)
\Edge[](v2)(v6)
\Edge[](v2)(v8)
\Edge[](v3)(v4)
\Edge[](v3)(v6)
\Edge[](v3)(v8)
\Edge[](v4)(v7)
\Edge[](v4)(v8)
\Edge[](v5)(v8)
\Edge[](v6)(v7)
\Edge[](v7)(v8)
\end{tikzpicture}
\end{center}

\label{ex:9verts}
\end{ex}

We remark $Z_G(t) = 1$ if and only if $G$
has no cycles, so $Z_G$ can determine whether a graph is a forest or not.
The girth $g$ of $G$ will be the smallest $\ell$ such that $a(\ell) \ne 0$.  So $Z_G$
determines $g$ and, if $G$ is connected, Scott and Storm \cite{scott-storm}
showed $Z_G$ determines the number of cycles of length $\ell$ for any $\ell < 2g$.

From \eqref{eq:hashimoto-det} it is clear that $Z_{G}(t)^{-1}$ is a polynomial in $t$ of
degree $2m^\pr$, where $m^\pr$ is the size of $G^\pr$.
In general, though $Z_G$ does not determine $m$ or $n$, it
does determine $m = m^\pr$ if $G$ is md2.  If we assume that $G$ is connected
(but not necessarily md2), then  we can say $Z_G$ determines $m-n$.
Namely, if $Z_G(t)=1$ then $G$ is a tree and $m-n=-1$.  Otherwise
$G^\pr$ is not the null graph so $m-n=m^\pr-n^\pr$, and  $m^\pr$ and $n^\pr$ are determined by $Z_G$, as we will see below.
Note that $Z_G$ does not determine $m-n$ if we drop the assumption that $G$ is connected,
since all forests have trivial zeta function.

\medskip
To get clean statements, for the rest of this section assume $G$ is md2 and connected.
For general connected $G$, the results here can be viewed as results about $G^\pr$ when $G^\pr$ is not the null graph.  Note that if $G$ is connected, then $G^\pr$ is also connected.

In this case $m \ge n$, and Hashimoto \cite{hashimoto} showed (i)
$Z_G(t)$ has a pole at $t=1$, which is of order $2= m-n+2$ if $G$ is a
circuit and $m-n+1$ otherwise, and (ii) $Z_G(t)$ has a pole at
$t=-1$ of order $m-n$, $m-n+1$ or $m-n+2$ according to whether
$G$ is non-bipartite, a bipartite non-circuit or a bipartite circuit.
Hashimoto \cite{hashimoto2}
later showed the residue at $t=1$ for non-circuits is
$2^{m-n+1}(n-m) \kappa(G)$, where $\kappa$ denotes complexity of $G$, i.e., the number of
spanning trees of $G$.
Note that $G$ is bipartite if and only if its geodesics all have even length, and recall
$G$ is a circuit if and only if $a(\ell) = 0$ for $\ell$ sufficiently large.

Consequently, $Z_G$ determines $n$, $m$, $\kappa(G)$, whether $G$ is bipartite, and whether $G$ is a circuit.
 In fact, since $\det(I-tA + t^2(D-I))$ is simply $\det(D+A)=\det |L| $ at $t=-1$,
\eqref{eq:bass-det} tells us $\lim_{t \to -1}  (1-t^2)^{n-m} Z_G(t)^{-1} = \det |L|$, which is
0 if and only if $G$ has a bipartite component.  So $Z_G(t)$ also
determines the product of eigenvalues of $|L|$.

Cooper \cite{cooper} also
showed $Z_G$  determines whether $G$ is regular, and if so, the degree of regularity
as well as the spectrum.  In fact if $G$ is $(q+1)$-regular, then the spectrum conversely
determines $Z_G$ by
\begin{equation}
Z_G(t) = (1-t^2)^{n-m} \prod_i (1- \lambda_i t + q t^2)^{-1}
\end{equation}
where $\{ \lambda_1, \ldots, \lambda_n \}$ are the eigenvalues of $A$
(attributed to A.\ Mellein \cite{czarneski}).

Write $V = \{ v_1, \ldots, v_n \}$, $d_i = \deg(v_i)$ and $q_i = d_i - 1$.
We see the leading term of $Z_G(t)^{-1}$ is $(\prod q_i) t^{2m}$.  Consequently, for $G$ md2, $Z_G(t)$ determines $\prod q_i$.
One also knows $2m = \sum d_i$, so knowing $Z_G(t)$
places rather strong restrictions on the degree sequence for md2 graphs.

In fact, Setyadi--Storm \cite{SS}, based on their enumeration of zeta functions of
connected md2 graphs on $\le 11$ vertices,
conjectured that connected md2 graphs with the same Ihara zeta function have
identical degree sequences.  However, we give a counterexample.

\begin{ex}
The graphs {\verb+K??CA?_FEcdk+} and {\verb+K??CA?_ccWNk+} on 12 vertices and 16 edges
drawn below (``the crab and the squid'')
have the same zeta function, but their degree sequences are (5, 5, 4, 2, 2, 2, 2, 2, 2, 2, 2, 2)
and (7, 3, 3, 3, 2, 2, 2, 2, 2, 2, 2, 2).

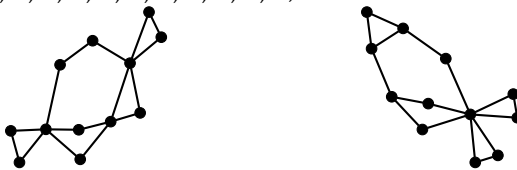
\begin{center}
\begin{tikzpicture}
\tikzset{VertexStyle/.style = {shape = circle, fill = black, inner sep = 0pt, outer sep = 0pt, minimum size = 4pt, draw} }
\useasboundingbox (0,0) rectangle (2.0cm,2.0cm);
\Vertex[L=\hbox{$.$},x=1.8369cm,y=2.0cm]{v0}
\Vertex[L=\hbox{$.$},x=1.0871cm,y=1.619cm]{v1}
\Vertex[L=\hbox{$.$},x=0.1146cm,y=0.0cm]{v2}
\Vertex[L=\hbox{$.$},x=1.7191cm,y=0.6586cm]{v3}
\Vertex[L=\hbox{$.$},x=0.9024cm,y=0.4341cm]{v4}
\Vertex[L=\hbox{$.$},x=0.9224cm,y=0.0408cm]{v5}
\Vertex[L=\hbox{$.$},x=2.0cm,y=1.666cm]{v6}
\Vertex[L=\hbox{$.$},x=0.6544cm,y=1.2985cm]{v7}
\Vertex[L=\hbox{$.$},x=0.0cm,y=0.4195cm]{v8}
\Vertex[L=\hbox{$.$},x=1.3274cm,y=0.5422cm]{v9}
\Vertex[L=\hbox{$.$},x=1.5846cm,y=1.3204cm]{v10}
\Vertex[L=\hbox{$.$},x=0.466cm,y=0.4402cm]{v11}
\Edge[](v0)(v6)
\Edge[](v0)(v10)
\Edge[](v1)(v7)
\Edge[](v1)(v10)
\Edge[](v2)(v8)
\Edge[](v2)(v11)
\Edge[](v3)(v9)
\Edge[](v3)(v10)
\Edge[](v4)(v9)
\Edge[](v4)(v11)
\Edge[](v5)(v9)
\Edge[](v5)(v11)
\Edge[](v6)(v10)
\Edge[](v7)(v11)
\Edge[](v8)(v11)
\Edge[](v9)(v10)
\end{tikzpicture}
\hspace{1in}
\begin{tikzpicture}
\tikzset{VertexStyle/.style = {shape = circle, fill = black, inner sep = 0pt, outer sep = 0pt, minimum size = 4pt, draw} }
\useasboundingbox (0,0) rectangle (2.0cm,2.0cm);
\Vertex[L=\hbox{$.$},x=0.0629cm,y=1.5131cm]{v0}
\Vertex[L=\hbox{$.$},x=1.9469cm,y=0.9089cm]{v1}
\Vertex[L=\hbox{$.$},x=1.4412cm,y=0.0cm]{v2}
\Vertex[L=\hbox{$.$},x=1.0494cm,y=1.3809cm]{v3}
\Vertex[L=\hbox{$.$},x=0.7386cm,y=0.4375cm]{v4}
\Vertex[L=\hbox{$.$},x=0.8166cm,y=0.7821cm]{v5}
\Vertex[L=\hbox{$.$},x=0.0cm,y=2.0cm]{v6}
\Vertex[L=\hbox{$.$},x=2.0cm,y=0.5947cm]{v7}
\Vertex[L=\hbox{$.$},x=1.7421cm,y=0.0933cm]{v8}
\Vertex[L=\hbox{$.$},x=0.4821cm,y=1.783cm]{v9}
\Vertex[L=\hbox{$.$},x=0.3309cm,y=0.8731cm]{v10}
\Vertex[L=\hbox{$.$},x=1.3762cm,y=0.6385cm]{v11}
\Edge[](v0)(v6)
\Edge[](v0)(v9)
\Edge[](v0)(v10)
\Edge[](v1)(v7)
\Edge[](v1)(v11)
\Edge[](v2)(v8)
\Edge[](v2)(v11)
\Edge[](v3)(v9)
\Edge[](v3)(v11)
\Edge[](v4)(v10)
\Edge[](v4)(v11)
\Edge[](v5)(v10)
\Edge[](v5)(v11)
\Edge[](v6)(v9)
\Edge[](v7)(v11)
\Edge[](v8)(v11)
\end{tikzpicture}
\end{center}

\label{ex:crabsquid}
\end{ex}

\subsection{Zeta functions of cones---(M1)}

The most obvious issue of using zeta functions to distinguish arbitrary graphs
is that no closed geodesics will pass through ``dangling links,'' i.e., paths not contained
in cycles.  To resolve this issue, an obvious thing to try is connecting all the degree 1
nodes to a new vertex.  However, doing this to just the degree 1 nodes is not a nice
operation on graphs (it is not injective), so it makes more sense to look at the
{\em cone} $G^*$ of $G$.  This is just the join of $G$ with a point: $G^* = G \vee K_1$,
where $\vee$ denotes the operation of graph join.  The new vertex in $G^*$ is denoted
$v_{n+1}$.

Clearly $G^*$ has order $n+1$ and increases the degree of each vertex in $G$ by 1.
That is, $G^*$ has degree sequence
$(d_i^*)$ where $d_i^* = d_i+1$ for $1 \le i \le n$ and $d_{n+1}^* = n$.  Hence if $G$ is
an md1 graph (a graph with no degree 0 vertices), then $G^*$ is an md2 graph.  A graph
$H$ on $n+1$ vertices will be the cone of some graph $G$ if and only if there is a vertex
of degree $n$, and $G$ can be recovered by deleting any vertex of degree $n$.

We propose to use $Z_{G^*}$ to study $G$.  Now any edge in $G$ will appear in
some closed geodesic in $G^*$, so it is reasonable to expect that $Z_{G^*}$ encodes
much of the structure of $G$.  Of course, degree 0 nodes in $G$ become degree 1
nodes in $G^*$ and still are not detected by $Z_{G^*}$.  If one wanted to, say,
count the degree 0 nodes only using zeta functions,
one could also look at the zeta function of the double cone $G^{**} = G \vee K_2$.
Instead, we will simply fix $n$, and see how to determine the number of degree 0 nodes from $n$ and $Z_{G^*}$.

\medskip
Hence, for the rest of this section, we will assume we know the order $n$ of $G$ (except
where stated otherwise)
and see what can be deduced about $G$ from $Z_{G^*}$.

We first show that we can determine the number of degree 0 nodes from $n$ and $Z_{G^*}$.
Let $G'$ be the graph obtained from $G$ by removing
all degree 0 nodes.  Note $G$ is the empty graph on $n$ vertices if and only if $Z_{G^*} = 1$.
Otherwise, $(G^*)^\pr = (G')^*$ and $(G')^*$ is connected md2,
so $Z_{G^*}$ determines the number of vertices and edges for $(G')^*$
and hence also for $G'$.  Combining this with knowing $n$, we can determine
$G$ from $G'$.  Hence in considerations below we may and will assume $G=G'$, i.e.,
$G$ has no degree 0 nodes, thus $G^*$ is connected md2.

From the previous section, we immediately see $Z_{G^*}$ determines the following:
$n$, $m$, $\prod d_i$, and the number of triangles in $G$, which equals
$\frac 12 a_{G^*}(3) - m$.  The latter follows as a 3-cycle in $G^*$ corresponds to
either a 3-cycle in $G$ or a triangle formed from an edge in $G$ with $v_{n+1}$.
Similarly, $a_{G^*}(4) = a_{G}(4) + 2\sum \binom{d_i}{2}$,
since the latter term is the number of directed paths of length 2 in $G$.
In particular, $Z_{G^*}$ determines $a_G(4) + \sum d_i^2$.
Also, one knows $a_{G^*}(5)$, which is $a_G(5)$ plus the number of directed paths of length
3 in $G$.

An elementary consequence of knowing $n$, $2m = \sum d_i$ and $\prod d_i$ is the following:
if we fix $a, b$ and $c$ and know that
each vertex has degree $a$, $b$ or $c$, then $Z_{G^*}$ determines
the degree sequence.
For general $G$, we remark that
one can determine the degree sequence by looking at sufficiently many
cones.  Namely, let $G^{*(r)} = G \vee K_r$.

\begin{lem} Let $G$ be a graph of possibly unknown order $n \ge 1$.  There exists a finite number $r$ such that if
$H_1, \ldots, H_{r}$ is any known sequence of graphs with distinct orders, then
$Z_{G \vee H_1}, \ldots, Z_{G \vee H_{r}}$ determines the degree sequence of $G$.  One may take $r$ to be at most the order of $G$ (or 3 for $G$ with order less than 3).  In particular, $Z_{G^*}, \ldots, Z_{G^{*(r)}}$ determines the  degree sequence of $G$.
\label{lem:ds}
\end{lem}
\begin{proof}
We will take $r$ to be at least 3, in order to guarantee that one of the $H_j$ has order at least 2 (we allow one of the $H_j$'s to be the null graph).  We will assume that we know the degree sequences of the $H_j$ and describe how the degree sequence of $G$ can be computed from $Z_{G \vee H_1}, \ldots, Z_{G \vee H_{r}}$.

Let $H$ be one of the $H_j$'s of order $h\geq 2$.  Then $G\vee H$ is connected and if $n\geq 2$ it is also md2.  It follows that the order of $(G\vee H)^\pr$ is less than $h+2$ if and only if $n=1$.  Otherwise it has order $h+n$ and, in either case, we can find $n$ from $Z_{G\vee H}$.  Assume from now on that $n\geq 2$.  Then $Z_{G\vee H}$ determines the number of edges of $G\vee H$ from which we can compute the number of edges of $G$, which we will call $m$.

Let $x_i$ be the number of vertices of $G$ of degree $i$ for $0\leq i \leq n-1$.  From $n$ and $m$ we have the two linear equations $\sum_{i=0}^{n-1} x_i = n$ and $\sum_{i=0}^{n-1} i x_i = 2m$.  Let $h_j$ be the order of $H_j$ for $1\leq j \leq r$.  First assume that each $h_j$ is at least 2 so $G\vee H_j$ are all connected md2 graphs.  Then $Z_{G\vee H_j}$ determines the products of the degrees of $G\vee H_j$ minus 1.  The degrees coming from vertices of $H_j$ are known and hence $Z_{G\vee H_j}$ will determine $\prod_{i=0}^{n-1} (i +h_j - 1)^{x_i}$, or equivalently, $\sum_{i=0}^{n-1} \ln(i+h_j-1)x_i$.  Thus we have a linear system with $r+2$ equations and $n$ unknowns.  The coefficient matrix $A$ of this linear system has $1, i, \ln(i+h_1 -1), \ln(i+h_2 - 1), ..., \ln(i+h_r - 1)$ as
the entries of its $i$-th column, for $0\leq i \leq n-1$.  If $\mathbf{c}= \bmx c_1 & c_2 & ... &c_{r+2}\emx$ is a vector with $\mathbf{c}A = \mathbf{0}$, then the function $f(t) = c_1 + c_2t+ c_3\ln(t + h_1-1) + c_4\ln(t+h_2-1)+...+ c_{r+2}\ln(t+h_{r}-1)$ has zeros at $0, 1, 2,...,n-1$.  As the $h_j$'s are distinct, this function changes direction at most $r$ times and thus has at most $r+1$ zeros unless all $c_i$ are zero.  If we take $r=n-2$ then the $c_i$ must all be 0 so $A$ is an invertible $n\times n$ matrix and the linear system has exactly one solution.  If we drop the assumption that the $h_j$'s are all at least 2, then at worst we can ignore $Z_G$ and $Z_{G^*}$ and take $r=n$.
\end{proof}

\begin{rem}
If we allow $n \ge 0$ in the above proposition, the order of $G$ can be determined using the $Z_{G\vee H_j}$'s, with one exception: if all the $H_j$'s are empty graphs, then the
$Z_{G\vee H_j}$ will not distinguish between the null graph and the graph on one vertex.  Also,
using fewer than $n$ joins often suffices to determine the degree sequence of $G$.  For
instance, the proof shows $n-2$ works if all $h_j$'s are $\ge 2$.  It also ignores the degree information gotten from $Z_{G^*}$ if one of the $H_j$ has order 1.  If we also take into account that the $x_i$ must be nonnegative integers, this could significantly decrease the number of joins needed in many cases.
\end{rem}

We also comment that for regular graphs, knowing $Z_{G^*}$ is
equivalent to knowing the spectrum of $G^*$ \cite{BS}.

\subsection{Zeta functions of graphs and their complements---(M2)}

We begin by thinking of another way to encode the dangling nodes or links
in the zeta function.  By a dangling node (link), we mean a vertex (edge)
not contained in any cycle, i.e., any node (edge) not in $G^\pr$.

First note that if we fix $n$, we know the number $n - n^\pr$ of dangling nodes,
since $Z_G$ determines the order $n^\pr$ of $G^\pr$.  In particular, $n$ and
$Z_G$ will tell us if $G$ is connected md2.

Recall that knowing the spectrum of $T$ gives slightly more information than
just $Z_G$---it also tells us $m$.  Since $Z_G$ tells us $m^\pr$, knowing
$\phi_T$ tell us $m-m^\pr$ which is the number dangling links.
However, this tells us nothing about the structure of dangling links---e.g.,
$\phi_T$ cannot distinguish among a cycle $C_n$ with leaves added, $C_n$ with a
path attached and $C_n$ disjoint union a forest, provided the number of edges
match.

If we want to somehow account for dangling nodes with zeta functions,
we can also try looking at both $G$ and $\bar G$.  Degree 0 and
degree 1 vertices in $G$ now have degree $n-1$ and $n-2$ in $\bar G$,
so for $n > 3$ they will now have degree at least 2.  However, it is still possible
that some of these vertices are dangling nodes in $\bar G$.  Nevertheless,
we can show the following.

\begin{prop} \label{prop:M2}
Let $G$ be a graph of order $n$.  Then at least one of the following holds:

(i) $G$ is determined by $\phi_T$ and $\phi_{\bar T}$ (or equivalently, by $Z_G$, $Z_{\bar G}$ and $m$); or

(ii) for any vertex $v \in V$, we have $v \in G^\pr$ or $v \in (\bar G)^\pr$.
\end{prop}

We remark this is not true if (i) is just replaced with ``$G$ is determined by
 $Z_G$ and $Z_{\bar G}$.''   For instance, let $G_1 = K_3 \sqcup \{ v \}$
and $G_2$ be the graph obtained from $G_1$ by adding an edge from $v$ to
one vertex in $K_3$. Then $\bar G_1$ is a tree and $\bar G_2$ is a forest,
so $Z_{G_1} = Z_{G_2} = Z_{K_3}$ and $Z_{\bar G_1} = Z_{\bar G_2} = 1$.
Thus we cannot distinguish $G_1$ and $G_2$ by looking at their zeta functions
and their complements' zeta functions.  Moreover, $v$ does not lie in $G_i^\pr$ or
$(\bar G_i)^\pr$ for $i = 1, 2$.

Before the proof, we give some numerical evidence that (M2)---using
$\phi_T$ and $\phi_{\bar T}$ (or $Z_G$, $Z_{\bar G}$ and $m$)
to distinguish graphs of order $n$---is much better than just using $Z_G$
and $Z_{\bar G}$, or just $Z_G$ or $\phi_T$.  This data is presented in
Table \ref{table1}.   The second column is the total number
 of graphs on $n$ nodes.  Subsequent columns contain the number of graphs $G$ on $n$
 nodes which are not determined respectively by $Z_G$, by $Z_G$ and $Z_{\bar G}$, by $\phi_T$,
and by both $\phi_T$ and $\phi_{\bar T}$.

\begin{table}[!h]
\small
\caption{Number of small graphs not distinguished by zeta invariants}
\centering
\begin{tabular}{r|r|r|r|r|r}
$n$  & \# graphs & $Z$ & $Z\bar Z$ & $T$ & $T \bar T$\\
\hline
2 &     2  & 2 & 2&  0 &    0 \\
3  &    4   & 3 & 2&0  &   0\\
4  &   11   & 8 & 4&4   &  0\\
5   &  34   & 23 & 8&15   &  0\\
6   & 156   & 94 & 22&75   &  0\\
7  & 1,044   & 534 & 68 &449  &   0\\
8 & 12,346  & 4,889 & 312 &4,297 &    0\\
9 & 274,668 & 76,807 & 350 &68,708  &   2
\end{tabular}
\label{table1}
\end{table}

The pair of graphs $(G_1, G_2)$ on 9 vertices with the same $\phi_T$ and $\phi_{\bar T}$
are the graphs with 18 edges pictured in Example \ref{ex:crabsquid}.  We remark that
$G_1 \simeq \bar G_2$ in this example.

One explanation for why (M2) is so effective is that in most cases it restricts the
problem of distinguishing arbitrary graphs with zeta functions to looking
just at connected md2 graphs, where zeta functions give us a lot of information
(cf.\ Section \ref{sec:ihara-prop}).
Similarly we can determine a lot about $G$ if $\bar G$ is connected md2.
And if neither $G$ nor $\bar G$ is connected md2, this places strong constraints
on $G$, as we will see in the proof.

We first treat a special case, where we
can say something stronger.

\begin{lem} Suppose $G$ is a forest of order $n$.  Then at least one
of the following holds:

(i) $\bar G$ is connected md2, or

(ii) $\phi_T$ and $\phi_{\bar T}$ determine $G$ among all graphs of order $n$.
\end{lem}

\begin{proof}
 By Table \ref{table1}, we see
$\phi_T$ and $\phi_{\bar T}$ determine $G$ for at least $n \le 8$.
Assume $n \ge 5$.  Then $G$ has at most 1 vertex of degree $\ge n-2$.  If it has
none, $\bar G$ is connected md2, so assume it does.  Then $G$ must be
one of the following three graphs: a star graph on $n$ nodes, a star-like graph
on $n$ nodes with exactly 1 node distance 2 from the hub (so the other $n-2$ non-hub nodes are
adjacent to the hub), or a star graph on $n-1$
nodes union a point.  The latter case can be distinguished from the previous two
by counting edges.  The former two can be distinguished by looking at the number
of edges in $(\bar G)^\pr$.  Since $Z_G$ detects forests, $\phi_T$ and $\phi_{\bar T}$
determine these 3 graphs among all graphs of order $n$.
\end{proof}

\begin{proof}[Proof of Proposition \ref{prop:M2}]
If $G$ or $\bar G$ is md2, then clearly $(ii)$ holds, so we may assume neither is md2.  Switching $G$ and $\bar G$ if necessary, we may also assume $G$ is connected.  In light of the lemma, we may further assume $G^\pr = (V^\pr, E^\pr)$ is not the null graph, and hence
 $G^\pr$ has order $n^\pr \geq 3$.

Let $H = (W, F)$ be the subgraph (which is a forest) of $G$ induced from $W=V-V^\pr$.  Note that, in $G$, no $w \in W$ can be adjacent to more than 1 vertex in $V^\pr$.  Also, in any connected component of $H$, there is exactly one vertex which is adjacent to a vertex in $V^\pr$.  Suppose that $H$ contains at least two vertices $w_1, w_2$.  If $w_1$ and $w_2$ are not adjacent in $G$ then there exists $v\in V^\pr$ such that neither $w_1$ nor $w_2$ is adjacent to $v$ in $G$ and hence $\bar G$ contains the triangle $w_1, w_2, v$.  If $w_1$ and $w_2$ are adjacent in $G$, then at most one of them is adjacent to a vertex in $V^\pr$ so there exist $v_1, v_2 \in V^\pr$ which are not adjacent to either $w_1$ or $w_2$.  Then $\bar G$ contains the 4-cycle $\{w_1,w_2\}\vee \{v_1,v_2\}$.  In either case, both $w_1$ and $w_2$ appear in $(\bar G)^\pr$.

It remains to consider the case where $|W|=1$.  Let $w$ be the unique vertex in $W$ and let $v$ be the unique vertex in $V^\pr$ which is adjacent to $w$ in $G$.  If any two vertices $v_1, v_2$ in $V^\pr -\{ v\}$ are not adjacent, then $w, v_1, v_2$ forms a triangle in $\bar G$.  Also, if there exists $v_1,v_2\in V^\pr$ which are not adjacent to $v$ then $\{w,v\}\vee \{v_1,v_2\}$ is a 4-cycle in $\bar G$.  In either case, $w$ appears in $(\bar G)^\pr$.  We may therefore assume the graph induced by $V^\pr -\{ v\}$ is a complete graph and there is a most 1 vertex in $V^\pr - \{v\}$ which does not connect to $v$.  This leaves us with 2 possibilities for $G$, but we see that $\bar G$ for these two possibilities are a star-like graph
on $n$ nodes with exactly 1 node distance 2 from the hub or a star graph on $n-1$ nodes union a point.  In the proof of the last lemma, we showed that these two graphs are distinguished by $\phi_T$ and $\phi_{\bar T}$ among all graphs of order $n$.
\end{proof}

Note that the proofs of the lemma and proposition show that for $n\geq 5$, there are exactly three pairs of graphs $(G, \bar G)$ of order $n$ for which it is not true that every vertex $v$ of $G$ is in $G^\pr$ or ${\bar G}^\pr$.  These are the following three graphs and their complements: the star graph on $n$ nodes, the star-like graph on $n$ nodes with exactly 1 node distance 2 from the hub, and the star graph on $n-1$ nodes union a point.  Hence in all but these three cases, every vertex of $G$ appears in a geodesic of $G$ or $\bar G$ (or both).

\section{The Bartholdi zeta function} \label{sec:bartholdi}

An alternative to using zeta functions of graphs related to $G$ in order to study $G$ is
to consider a more general notion of zeta function which involves dangling nodes and links.
If we think about the adjacency spectrum, or equivalently the closed walk spectrum, it can
distinguish things like path graphs from star graphs because backtracking is allowed in closed
walks.  On the other hand, the closed walk spectrum loses a lot of information contained in the geodesic
length spectrum.  Bartholdi \cite{bartholdi} introduced a more general zeta function which encodes
 both the closed walk spectrum and the geodesic length spectrum.

Let $\gamma = (e_1, \ldots, e_\ell)$ be a closed walk of length $\ell = \ell(\gamma)$.  The number of backtracks in $\gamma$
is the number of $1 \le i < \ell$ such that $e_{i+1} = e_i^{-1}$.  We say $\gamma$ has a tail if $e_\ell = e_1^{-1}$.
The cyclic bump count $\cbc(\gamma)$ is the number of backtracks in $\gamma$ plus 1 or 0, according to whether
$\gamma$ has a tail or not.  The cyclic permutation group $\langle \sigma \rangle$ defined in Section \ref{sec:ihara} acts on closed walks of
length $\ell$ and preserves the cyclic bump count.  A closed walk is primitive if it is not of the form $k\gamma$ for $k > 1$.
Let $a(\ell; c) = a_G(\ell;c)$ denote the number of $\langle \sigma \rangle$ orbits of primitive closed walks in
$G$ with $\ell(\gamma) = \ell$ and $\cbc(\gamma) = c$.  Note $a(\ell; 0) = a(\ell)$ since $\cbc(\gamma) = 0$ means
$\gamma$ is a geodesic.
The {\em Bartholdi zeta function} is
\begin{equation} \label{eq:barth-def}
 \barth_G(t, u) = \prod_\gamma (1-u^{\cbc(\gamma)} t^{\ell(\gamma)} )^{-1} = \prod_{c, \ell} (1 - u^c t^\ell)^{-a(\ell; c)}
 = \exp \left( \sum_{c, \ell} a(\ell; c) \sum_{k\geq 1} \frac{u^{ck} t^{\ell k}}k \right),
\end{equation}
where $\gamma$ runs over $\langle \sigma \rangle$ equivalence classes of primitive closed walks in
$G$.   Note this gives the Ihara zeta function when $u=0$: $\barth_G(t,0) = Z_G(t)$.

Bartholdi \cite{bartholdi} proved an analogue of the Bass determinant formula:
\begin{equation} \label{eq:barth-bass}
\barth_G(t,u) = (1 - (1-u)^2t^2)^{n-m} \det( I - tA + (1-u)(D - (1-u)I)t^2)^{-1}.
\end{equation}

\subsection{Properties determined by the Bartholdi zeta function---(M3)}

First observe \eqref{eq:barth-def} tells us
that knowing $\barth_G$ is equivalent to knowing all of the numbers $a(\ell; c)$.
Since $a(2; 1) = 2m$, $\barth_G$ determines $m$.  However, $\barth_G$ does not determine $n$
as adding isolated vertices does not change $\barth_G$.
We see $\mathcal Z_G$ determines the
number of 3-, 4-, and 5-cycles in $G$ since $Z_G$ does.  However it does not determine the
number of 6-cycles, as the pair of graphs in Example \ref{ex:crabsquid} have the same Bartholdi
zeta function.

\cite{KL} observed that one can rewrite \eqref{eq:barth-bass}
in terms of the generalized characteristic polynomial $\phi_{AD}^G(\lambda, x) = \det(\lambda I - A + xD)$ by
\begin{equation} \label{eq:phiAD}
\barth_G(t,u) = (1 - (1-u)^2t^2)^{n-m} t^{-n} \phi_{AD}^G( t^{-1} - (1-u)^2t, (1-u)t )^{-1},
\end{equation}
We will write $\phi_{AD} = \phi_{AD}^G$ if the graph $G$ is clear from context.
It is stated in \cite{WLLX} that $\barth_G(t,u)$ determines $\phi_{AD}$ and vice versa,
but this is not true without further qualification in the same way that $Z_G$ does not determine
$\phi_T$.
Namely, $\phi_{AD}$ determines $m$ and $n$, so also determines $\barth_G$.
On the other hand, $\barth_G$ does not determine $n$.  However, since $\barth_G$
determines $m$
we see that $\barth_G$ and $n$ determine $\phi_{AD}$ and vice versa.

From now on, we now consider (M3):  what can be determined from $n$ and $\barth_G$,
or equivalently, $\phi_{AD}$?

By specializing $\phi_{AD}(\lambda, x)$ to $x = 0, \pm 1$,
we see $\phi_{AD}$ determines the spectra of $A$, $L$ and
$|L|$.  There is much literature about what these spectra individually determine about $G$,
and various families of graphs that are determined by such spectra
(e.g., see the books \cite{BH}, \cite{CRS} and the survey articles \cite{vDH1}, \cite{vDH2}).
We just recall a few things determined by knowing all of these spectra: the number of edges,
regularity,
the number of components, the number of bipartite components, the complexity and the closed
walk spectrum.

From $Z_G$, we also know whether $G$ is connected md2, and all of the things
discussed in Section \ref{sec:ihara-prop}.
In addition, \cite{WLLX} proves that $\phi_{AD}$ determines the degree sequence of $G$.

\subsection{Bartholdi zeta functions of complements---(M4)} \label{sec:M4}

Finally, consider our last method (M4): what can be determined from $n$, $\barth_G$
and $\barth_{\bar G}$.

Equation \eqref{eq:phiAD} tells us
 these quantities determine the spectra of $A+xD$ and $A+xD-J$ for all $x$.  When $x=0$,
\cite{JN} showed this determines the spectra of $A + yJ$ for all $y$, and the proof
(see \cite{vDH1} for a simpler proof) in fact works for $x \ne 0$,
i.e., $n$, $\barth_G$ and $\barth_{\bar G}$ determine the generalized characteristic
polynomial $\phi_{ADJ}(\lambda, x, y) = \phi_{ADJ}^G(\lambda,x,y) = \det(\lambda I - A + xD +y J)$.  The converse,
that $\phi_{ADJ}$ determines $n$, $\barth_G$ and $\barth_{\bar G}$, is straightforward.

Consequently, (M4) determines everything (M2) and (M3) do.  In fact, we show
below that (M4) determines everything (M1) does.
\begin{lem}
Let $X$ and $Y$ be $n\times n$ and $m\times m$ matrices respectively and let $J_{nm}$ denote the $n\times m$ all ones matrix.  Let $M$ be the block matrix $M=\bmx X & J_{nm} \\ J_{nm} & Y \emx$.  We will write $J$ for the square all ones matrix when the order is clear from context.    The
following are true:
 \begin{enumerate}
 \item
 The spectra of $M$ and $J-M$ are determined by the spectra of $X, J-X, Y$, and $J-Y$.
 \item
 The spectra of $X$ and $J-X$ are determined by the spectra of $M$, $J-M$, $Y$, and $J-Y$.
 \end{enumerate}
\end{lem}
\begin{proof}
Note that, for an $r \times r$ matrix, one can determine the spectrum of the matrix from the traces of the first $r$ powers of the matrix and, conversely, one can determine the traces of all powers from the spectrum.

We can prove using induction that $M^k$ has the form $\bmx X^k + A_k & B_k\\ C_k & Y^k+D_k\emx$ where the matrices $A_k,B_k,C_k,D_k$ are sums of
matrices which are the product of $k$ matrices coming from $\{X,Y, J_{nm}, J_{mn}\}$.  For $k\geq 2$, each product of $k$ matrices appearing in $A_k$ and $D_k$ is such that at most $k-2$ of the $k$ matrices are $X$'s or $Y$'s, and for $B_k$ and $C_k$ at most $k-1$ are $X$'s or $Y$'s.

From the equation for $M^k$, we get that $\tr(M^k) = \tr(X^k)+\tr(Y^k) +\tr(A_k)+\tr(D_k)$.  We also use the following properties:
$\tr(XJ_{nm}YJ_{mn})= (\sum_{i,j} x_{ij}) (\sum_{i,j} y_{ij}) =  \tr(XJ)\tr(YJ)$, trace is invariant under cyclic permutations of products, and the product of any two all ones matrices is a scalar multiple of an all ones matrix.  From these properties, it follows that $\tr(A_k)$ and $\tr(D_k)$ can be determined from $\tr(X^iJ)$, and $\tr(Y^iJ)$ for $i=0,1,2,..,k-2$.  This leads to a relation of the form $\tr((J-X)^k) = (-1)^k(\tr(X^k) -k\tr(X^{k-1}J) + \cdots)$, where the omitted terms are determined by $\tr(X^iJ)$ for $i=0,1,2,...,k-2$, as well as a similar relation for $Y$.

We now prove the two statements.  Note that in both cases we can determine $n$ and $m$.  Also note that $J-M$ is the block diagonal matrix with $J-X$ and $J-Y$ on the diagonals, so any two of the spectra of $J-M, J-X,$ and $J-Y$ determine the third.

First suppose we know the spectra of the matrices $X, J -X, Y$, and $J-Y$.  Then we know $\tr((J-X)^i)$ and $\tr(X^i)$ for all $i$.  From this, we can recursively determine $\tr(X^iJ)$ for all $i$, and similarly for $Y$.  We can therefore compute $\tr(M^k) = \tr(X^k)+\tr(Y^k) +\tr(A_k)+\tr(D_k)$ for all $k$ and thus determine the spectrum of $M$.

Suppose now that we know the spectra of $Y, J-Y, M, J-M$.  As mentioned above, this tells us the spectrum of $J-X$.  It remains to show that we can compute $\tr(X^k)$ for $k=1,2,...,n$.  For $k=1,2$ we can find $\tr(X)$ and $\tr(X^2)$ from $\tr(M)=\tr(X)+\tr(Y)$ and $\tr(M^2)= \tr(X^2)+\tr(Y^2)+\tr(J_{mn}J_{nm})+\tr(J_{nm}J_{mn})$.  If $k>2$ and we know $\tr(X^i)$ for $i<k$, then from $\tr((J-X)^i)$ and $\tr(X^i)$ for $i=1,2,..,k-1$ we can find $\tr(X^iJ)$ for $i=1,2,...,k-2$, and
therefore also $\tr(A_k)$ and $\tr(D_k)$.  From these, we can compute $\tr(X^k)$ from $\tr(M^k) = \tr(X^k)+\tr(Y^k) +\tr(A_k)+\tr(D_k)$.
\end{proof}

If $D$ is the degree matrix of $G$, we write $\bar D$ for the degree matrix of $\bar G$.

\begin{thm} Fix $x \in \C$.
Let $G_1$ and $G_2$ be two graphs of the same order and let $H$ be any graph.  The graphs $G_1\vee H$ and $G_2\vee H$ are cospectral with respect to $A+xD$ and $\bar A + x\bar D$ if and only if $G_1$ and $G_2$ are cospectral with respect to  $A+xD$ and $\bar A+ x\bar D$.

\label{thm:cone-comp}
\end{thm}
\begin{proof}
Let $M=A+xD$ and note that, for an order $r$ graph, $\bar M := \bar A + x \bar D = J-M + (x(r-1)-1)I$.  Thus if we know $x$ and $r$, knowing the spectrum of $\bar M$ is equivalent to knowing the spectrum of $J-M$.  Let $n$ be the order of $G_1$ and $G_2$, let $m$ be the order of $H$, and denote their respective matrices $M=A+xD$ by $M_{G_i}$ and $M_H$.  The join $G_i\vee H$ has matrix $M_{G_i\vee H} = \bmx M_{G_i} + xm I_n & J_{nm}\\ J_{mn}  & M_H + xnI_m \emx$.  We can thus apply the previous lemma to show that if $G_1$ and $G_2$ are cospectral with respect to $M$ and $\bar M$, then so are $G_1\vee H$ and $G_2\vee H$ and vice versa.
\end{proof}

This is already well known in the case that $x=-1$, which corresponds to the
Laplacian---in fact less is needed in this case since the Laplacian spectrum of $G$ determines
that of $\bar G$.

\begin{cor} \label{cor:cone-comp}
The graphs $G_1$ and $G_2$ have the same spectra with respect to $A$ and $\bar A$
(or $|L|$ and $\overline{|L|}$),
if and only if the same is true for $G_1^*$ and $G_2^*$.  Similarly, $G_1$ and $G_2$ have the same generalized characteristic polynomial $\phi_{ADJ}$ if and only if the same is true for $G_1^*$ and $G_2^*$, i.e., $\phi_{ADJ}^{G_1} = \phi_{ADJ}^{G_2}$ if and only if $\phi_{ADJ}^{G_1^*} = \phi_{ADJ}^{G_2^*}$.
\end{cor}

\section{Constructions} \label{sec:constructions}

Here we discuss three well-known constructions of cospectral graphs, and show that
in many cases, such graphs are distinguished by methods (M1)--(M4) using constraints on
degree distributions.  This provides some evidence for our conjectures in the next section
that most cospectral graphs can be distinguished by any of (M1)--(M4).

In the last part of this section, we give a new construction for graphs which cannot
be distinguished by (M1)--(M4), which generalizes the pair of graphs from Example \ref{ex:9verts}.

\subsection{GM switching}

Godsil and McKay present a method for constructing cospectral pairs of graphs \cite{GM}, which is now referred to as GM switching.  We say $G$ satisfies the $(k+1)$-GM (or just GM) condition if there is an ordering of the vertices such
that the adjacency matrix can be written in the form
$$ A = \bmx B_1 & B_{12} & \cdots & B_{1k} & N_1 \\ {}^t B_{12} & B_2 & \cdots & B_{2k} & N_2 \\
\vdots & &\ddots & & \vdots \\ {}^t B_{1k} & \cdots & {}^t B_{k-1,k} &  B_k & N_k \\
{}^t  N_1 & \cdots & {}^t N_{k-1} & {}^t N_k & C \emx$$
where (i) each $G_{B_i}$ is regular, (ii) each $B_{i,j}$ has constant row and column sums, and
(iii) each column of each $N_i$ has exactly 0, $b_i/2$ or $b_i$ 1's.  Here $b_i$ is the order of each square matrix $B_i$, and we assume
at least one of the $b_i$'s is even.  Note $k+1$ is the number of diagonal blocks, so the ordering of
the vertices and the size of the $B_i$'s determines a partition of the vertex set into $k+1$ subsets
$V_{B_1}, \ldots, V_{B_k}, V_C$.

Let $\~N_i$ be the matrix formed from $N_i$ by replacing each column $\mathbf{v}$ of $N_i$ which consists of $b_i/2$ ones by the column $J_{b_i,1}-\mathbf{v}$.  The GM switch of $G$ is the graph $\~G$ with adjacency matrix $\~A$ which is gotten from $A$ by replacing each $N_i$ with $\~N_i$.  Then $G$ and $\~G$ are cospectral with respect to the adjacency matrix.  The proof is to exhibit a matrix $Q$ which conjugates $A$ to $\~A$ and has constant row and column sums of 1, and hence commutes with the all ones matrix $J$.  It follows that the complements of $G$ and $\~G$ are also cospectral with respect to the adjacency matrix.  Also, by Corollary \ref{cor:cone-comp}, the cones of $G$ and $\~G$ are cospectral with respect to the adjacency matrix.

Haemers and Spence introduced a special case of GM switching \cite{HS}, which they called GM* switching, which gives cospectral pairs with respect to any matrix of the form $A+xD$.  They only define GM* switching in the case where $k=1$, but it works in the more general setup as well.  We say $G$ satisfies the $(k+1)$-GM* condition if $G$ satisfies the $(k+1)$-GM condition with the additional requirement that each vertex in $G_{B_i}$ has the same degree in $G$ (i.e., each $N_i$ has constant row sums).  This condition guarantees that $G$ and $\~G$ have the same degree matrix $D$. Also $D$ and $J$ commute with $Q$, the matrix which conjugates $A$ to $\~A$, and hence $A+xD+yJ$ and $\~A+xD+yJ$ are cospectral for any $x,y$ so $G$ and $\~G$ have the same generalized characteristic polynomial $\phi_{ADJ}$.  Also, by Corollary \ref{cor:cone-comp}, the cones of $G$ and $\~G$ have the same $\phi_{ADJ}$.  We call this $(k+1)$-GM* switching and say that
$(G, \~G)$ is a $(k+1)$-GM* pair provided $G \not \simeq \~G$.

In the next section, we will find that GM* switching accounts for a significant percentage of graphs up to 11 vertices which have the same $\phi_{ADJ}$
(see Table \ref{tab:main}).  For the graphs on up to 10 vertices, 3-GM* switching does
not produce any such examples that 2-GM* does not.
However for 11 vertices, there are 108 3-GM* pairs which cannot be obtained by 2-GM* switching,
even if one allows successive 2-GM* switching.  (It happens that $(G_1, G_2)$ can be a 3-GM*
pair but not a 2-GM* pair, but that $G_1$ and $G_2$ are both obtained as 2-GM* switches
from a third graph $G_3$.)

\begin{ex} \label{ex:3GMst}
The graphs {\verb+J?BD?oX[F[?+} and {\verb+J?`CP``LE{?+ } on 11 vertices and 18 edges
drawn below are a 3-GM* pair, but not a 2-GM* pair.  Further, neither of these
form a 2-GM* pair with any other graph.  The subgraphs $G_{B_1}, G_{B_2}$, and $G_C$ are the subgraphs induced by vertices 1--6, 7--9, and 10--11 respectively.
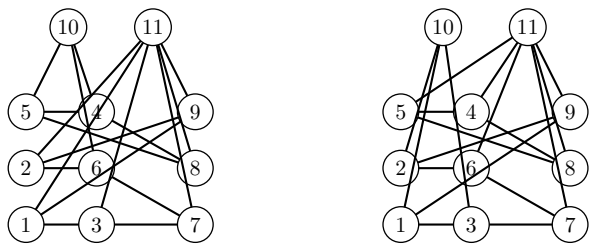
\begin{center}
\begin{tikzpicture}[scale=0.75,transform shape]
%
%
\useasboundingbox (0,0) rectangle (3.0cm,3.75cm);
\Vertex[L=\hbox{$1$},x=0.0 cm,y=0.0cm]{v0}
\Vertex[L=\hbox{$3$},x=1.25cm,y=0.0cm]{v2}
\Vertex[L=\hbox{$2$},x=0.0cm,y=1.0cm]{v1}
\Vertex[L=\hbox{$6$},x=1.25cm,y=1.0cm]{v5}
\Vertex[L=\hbox{$4$},x=1.25cm,y=2.0cm]{v3}
\Vertex[L=\hbox{$5$},x=0cm,y=2.0cm]{v4}
\Vertex[L=\hbox{$7$},x=3cm,y=0cm]{v6}
\Vertex[L=\hbox{$8$},x=3cm,y=1.0cm]{v7}
\Vertex[L=\hbox{$9$},x=3cm,y=2.0cm]{v8}
\Vertex[L=\hbox{$10$},x=0.75cm,y=3.5cm]{v9}
\Vertex[L=\hbox{$11$},x=2.25cm,y=3.5cm]{v10}
\Edge[](v0)(v2)
\Edge[](v0)(v8)
\Edge[](v0)(v10)
\Edge[](v1)(v5)
\Edge[](v1)(v8)
\Edge[](v1)(v10)
\Edge[](v2)(v6)
\Edge[](v2)(v10)
\Edge[](v3)(v4)
\Edge[](v3)(v7)
\Edge[](v3)(v9)
\Edge[](v4)(v7)
\Edge[](v4)(v9)
\Edge[](v5)(v6)
\Edge[](v5)(v9)
\Edge[](v6)(v10)
\Edge[](v7)(v10)
\Edge[](v8)(v10)
\end{tikzpicture}
\hspace{1in}
 \begin{tikzpicture}[scale=0.75,transform shape]
%
%
\useasboundingbox (0,0) rectangle (3.0cm,3.5cm);
\Vertex[L=\hbox{$1$},x=0.0 cm,y=0.0cm]{v0}
\Vertex[L=\hbox{$3$},x=1.25cm,y=0.0cm]{v2}
\Vertex[L=\hbox{$2$},x=0.0cm,y=1.0cm]{v1}
\Vertex[L=\hbox{$6$},x=1.25cm,y=1.0cm]{v5}
\Vertex[L=\hbox{$4$},x=1.25cm,y=2.0cm]{v3}
\Vertex[L=\hbox{$5$},x=0cm,y=2.0cm]{v4}
\Vertex[L=\hbox{$7$},x=3cm,y=0cm]{v6}
\Vertex[L=\hbox{$8$},x=3cm,y=1.0cm]{v7}
\Vertex[L=\hbox{$9$},x=3cm,y=2.0cm]{v8}
\Vertex[L=\hbox{$10$},x=0.75cm,y=3.5cm]{v9}
\Vertex[L=\hbox{$11$},x=2.25cm,y=3.5cm]{v10}
\Edge[](v0)(v2)
\Edge[](v0)(v8)
\Edge[](v0)(v9)
\Edge[](v1)(v5)
\Edge[](v1)(v8)
\Edge[](v1)(v9)
\Edge[](v2)(v6)
\Edge[](v2)(v9)
\Edge[](v3)(v4)
\Edge[](v3)(v7)
\Edge[](v3)(v10)
\Edge[](v4)(v7)
\Edge[](v4)(v10)
\Edge[](v5)(v6)
\Edge[](v5)(v10)
\Edge[](v6)(v10)
\Edge[](v7)(v10)
\Edge[](v8)(v10)
\end{tikzpicture}
\end{center}
\end{ex}

We now focus on the case of 2-GM switching.  For simplicity, we omit the subscripts and
write the adjacency matrix as
$$A = \bmx  B & N\\ {}^t N & C \emx. $$
For $G \not \simeq \~G$, we need $B$ to have even size $\ge 4$. Further, the larger $B$
is, the less likely it is that the GM condition will be satisfied.
So Haemers and Spence \cite[Thm 3]{HS} use GM switching with B of size 4 to get a lower
bound on the number of non-DS graphs.  The following shows that most cospectral pairs thus
constructed are distinguished by zeta functions.

\begin{thm} \label{thm:GM-zeta}
Suppose $G_B$ is regular on $4$ vertices and $G_C$ is an md1 (resp.\ md2)
graph on $n$ vertices.   Then the proportion of GM-admissible choices of $N$ such that
the pair $(G, \~G)$ of {\em labeled} graphs
formed by GM switching on $(B,C,N)$ which are distinguished by $Z^*$ (resp.\ $Z$)
goes to $1$ as $n \to \infty$.
\end{thm}

\begin{proof} For each column of $N$, there are $2 + \binom{4}{2} = 8$ possible choices, which we
count with equal probability.  The probability that at least 2 columns are all ones is
\[ 1 - \left[\binom{n}{0} \left(\frac 18\right)^0 \left(\frac 78\right)^n +  \binom{n}{1} \left(\frac 18\right)^1 \left(\frac 78\right)^{n-1}\right] = 1 - \left(\frac{7+n}8 \right)\left( \frac 7 8 \right)^{n-1}. \]
When this is satisfied, $G$ and $\~G$ are md1 if $G_C$ is md1 and md2 if $G_C$ is md2.
Since this probability tends to 1 as $n \to \infty$,
from now on, we will assume $G$ and $\~G$ are md1 in the case of $Z^*$ or md2 in the case of
$Z$.

If $G$ and $\~G$ are not distinguished by $Z^*$ (resp.\ $Z$), then the products of the vertex degrees (resp.\ degrees minus 1)
must be the same.  Since all vertices coming from $G_C$ have the same degree, this is equivalent to knowing that the products
of the degrees of vertices (resp.\ degrees minus 1) from $B$ are the same.  Let $k$ be the degree of a vertex in $G_B$ plus the number
of all ones columns in $N$ in the case of $Z^*$, or this number minus 1 in the case of $Z$.  Let $N_0$ be the submatrix of $N$ formed by removing the columns consisting of all zeroes or all ones.  Let $x_i$ denote the number of 1's in the $i$-th row of $N_0$.  Then $G$ and $\~G$
being distinguished by $Z^*$ or $Z$ implies
\begin{equation} \label{GM-degcond}
 \prod_{i=1}^4 (k+x_i) = \prod_{i=1}^4 (k+n_0 - x_i),
\end{equation}
where $n_0$ is the number of columns in $N_0$.

View $n_0$ as fixed for now.  Since $x_4 = 2n_0 - x_1 - x_2 - x_3$, for fixed $x_1, x_2$, the solutions to \eqref{GM-degcond} in $x_3$
are the solutions to a degree 2 polynomial in $x_3$, of which there are at most 2.  Now view the top 2 rows of $N_0$ as fixed.
Let $r$ be the number of columns in $N_0$ which have exactly one 1 in the first 2 entries. Then there are $r+1$ choices for $x_3$.  There are $2^r$ choices for row 3 of $N_0$, hence for an integer $0\leq y \leq r$, the probability that $x_3=y$ is $\binom{r}{y} \frac 1{2^r}\leq \binom{r}{\lceil r/2 \rceil}\frac 1{2^r}$.
Given $r \ge 3$, the probability that $x_3$ avoids solving \eqref{GM-degcond} is at least
\[    1 - 2 \binom{r}{\lceil r/2 \rceil} \frac 1{2^r} \ge 1 - \frac{2}{\sqrt{3 {\lceil r/2 \rceil} + 1}} \ge 1 - \sqrt{\frac 8{3r}}. \]
(Here we use the inequality $\binom{2n}{n} \le \frac{2^{2n}}{\sqrt{3n+1}}$, which implies
$\binom{r}{\lceil r/2 \rceil} \le \frac{2^r}{\sqrt{3 {\lceil r/2 \rceil} + 1}}$.)
Also note
\[ P( width(N_0) = n_0) = \binom{n}{n_0} \left(\frac 34\right)^{n_0} \left(\frac 14\right)^{n-n_0} \]
and for a fixed $n_0$ the probability that the number of columns in $N_0$ which have exactly one 1 in the first 2 entries is $r$ is
\[P(r|n_0) = \binom{n_0}{r} \left(\frac 23\right)^{r} \left(\frac 13\right)^{n_0-r}. \]

Fix $0 < \delta < 1$.  By the law of large numbers, for any $\epsilon>0$, the probability that $r\ge n_0 ( \frac 23 -\epsilon)$ goes to 1 as $n_0 \to \infty$.  Hence the probability that $r \ge n_0^\delta$ goes to 1 as $n_0 \to \infty$. Similarly, the probability that $n_0 \ge n^\delta$ goes
to 1 as $n_0 \to \infty$.   So the probability that $x_3$ avoids the solutions of \eqref{GM-degcond} is at least
\[ P(n_0 \ge n^\delta) P(r \ge n_0^\delta) \left(1 - \sqrt{\frac 8{3n_0^\delta}}\right) \ge P(r \ge n^{\delta^2} ) \left(1 - \sqrt{\frac 8{3n^{\delta^2}}} \right), \]
which goes to 1 as $n \to \infty$.
\end{proof}

Let $g_n$ denote the number of simple graphs (up to isomorphism) of order $n$.

\begin{cor} The number of graphs $G$ on $n$ vertices for which there exists a cospectral nonisomorphic graph $\~G$ but
$Z_{\vphantom{\~G}G} \ne Z_{\~G}$ is at least $n^3 g_{n-1} ( \frac 1{24} - o(1) )$.  The same statement is true with $Z_{\vphantom{\~G}G} \ne Z_{\~G}$ replaced by
$Z_{\vphantom{\~G^*}G^*} \ne Z_{\~G^*}$.
\end{cor}

\begin{proof}
Haemers and Spence \cite[Thm 3]{HS} show there are at least $n^3 g_{n-1} ( \frac 1{24} - o(1) )$ non-isomorphic pairs $(G, \~G)$ obtained from
GM switching with $B$ of size 4.  Further, for almost all of these pairs, both graphs are md2.  Since, in the  argument above,
we can replace the condition that $C$ is md2 or md1 with knowing $G$ and $\~G$ are md2 or md1, we see almost all such pairs
are distinguished by $Z$ or $Z^*$.
\end{proof}

\subsection{Coalescence construction}

Suppose $G_1$ and $G_2$ are two graphs of order $n$ with the same adjacency spectra but different degree sequences.
Assume there are vertices $x_1$ of $G_1$ and $x_2$ of $G_2$ such that $G_1 - \{ x_1 \}$ and $G_2 - \{ x_2 \}$ are cospectral.  Let $U_i$ be the vertex set of $G_i - \{ x_i \}$.
Let $\Gamma$ be any graph with a fixed vertex $y$.  Let $G_i'$ be the coalescence of $(G_i, x_i)$ with $(\Gamma, y)$, i.e., the
union of $G_i$ and $\Gamma$ after identification of $x_i$ with $y$.

\begin{prop}  With notation as above, $\Spec(G_1') = \Spec(G_2')$ for any $(\Gamma, y)$.   If $G_1$ and $G_2$  are md2 and $\prod_{v \in U_2} (\deg(v) -1) \ne \prod_{v \in U_2} (\deg(v) - 1)$,
then $Z_{G_1'} \ne Z_{G_2'}$.  Similarly, if $G_1$ and $G_2$ are md1 and
$\prod_{v \in U_1} \deg(v) \ne \prod_{v \in U_2} \deg(v)$, then $Z_{G_1'^*} \ne Z_{G_2'^*}$.
\end{prop}

\begin{proof} The cospectrality is due to Schwenk \cite{schwenk}.

For the distinction by zeta functions, note that the fact that $G_1 - \{ x_1 \}$ and $G_2 - \{ x_2\}$ are cospectral implies that
$\deg_{G_1}(x_1) = \deg_{G_2}(x_2)$ as the spectrum determines the number of edges.  Note $\deg_{G_i}(v) = \deg_{G_i'}(v)$
for any $v \in U_i$ and $\deg_{G_1'}(v) = \deg_{G_2'}(v)$ for any vertex $v$ of $\Gamma$.
The assertions follow as $Z_G$ (resp.\ $Z_{G^*}$) determines the product of the degrees minus 1 (degrees)
of the pruned graph $G^\pr$.
\end{proof}

\begin{ex} Let $G_1$ and $G_2$ be the graphs  {\verb+F?zPw+} and {\verb+F@Rfo+}
on 7 vertices and 10 edges pictured below, where $x_1$
and $x_2$ are the white vertices.
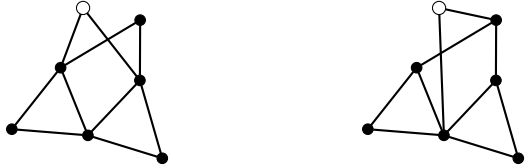
\begin{center}
\begin{tikzpicture}
\tikzset{VertexStyle/.style = {shape = circle, inner sep = 0pt, outer sep = 0pt, minimum size = 5pt, draw} }
\useasboundingbox (0,0) rectangle (2.0cm,2.25cm);
\Vertex[style={draw=black,text=white, shape=circle},L=\hbox{$.$},x=0.9469cm,y=2.0cm]{v0}
\tikzset{VertexStyle/.style = {shape = circle, fill = black, inner sep = 0pt, outer sep = 0pt, minimum size = 4pt, draw} }
\Vertex[L=\hbox{$.$},x=1.7058cm,y=1.838cm]{v1}
\Vertex[L=\hbox{$.$},x=0.0cm,y=0.3868cm]{v2}
\Vertex[L=\hbox{$.$},x=2.0cm,y=0.0cm]{v3}
\Vertex[L=\hbox{$.$},x=0.6482cm,y=1.2049cm]{v4}
\Vertex[L=\hbox{$.$},x=1.7025cm,y=1.0359cm]{v5}
\Vertex[L=\hbox{$.$},x=1.0112cm,y=0.3051cm]{v6}
\Edge[](v0)(v4)
\Edge[](v0)(v5)
\Edge[](v1)(v4)
\Edge[](v1)(v5)
\Edge[](v2)(v4)
\Edge[](v2)(v6)
\Edge[](v3)(v5)
\Edge[](v3)(v6)
\Edge[](v4)(v6)
\Edge[](v5)(v6)
\end{tikzpicture}
\hspace{1in}
\begin{tikzpicture}
\tikzset{VertexStyle/.style = {shape = circle, inner sep = 0pt, outer sep = 0pt, minimum size = 5pt, draw} }
\useasboundingbox (0,0) rectangle (2.0cm,2.0cm);
\Vertex[style={draw=black,text=white, shape=circle},L=\hbox{$.$},x=0.9469cm,y=2.0cm]{v0}
\tikzset{VertexStyle/.style = {shape = circle, fill = black, inner sep = 0pt, outer sep = 0pt, minimum size = 4pt, draw} }
\Vertex[L=\hbox{$.$},x=1.7058cm,y=1.838cm]{v5}
\Vertex[L=\hbox{$.$},x=0.0cm,y=0.3868cm]{v4}
\Vertex[L=\hbox{$.$},x=2.0cm,y=0.0cm]{v3}
\Vertex[L=\hbox{$.$},x=0.6482cm,y=1.2049cm]{v1}
\Vertex[L=\hbox{$.$},x=1.7025cm,y=1.0359cm]{v2}
\Vertex[L=\hbox{$.$},x=1.0112cm,y=0.3051cm]{v6}
\Edge[](v0)(v5)
\Edge[](v0)(v6)
\Edge[](v1)(v4)
\Edge[](v1)(v5)
\Edge[](v1)(v6)
\Edge[](v2)(v3)
\Edge[](v2)(v5)
\Edge[](v2)(v6)
\Edge[](v3)(v6)
\Edge[](v4)(v6)
\end{tikzpicture}
\end{center}
Then $G_1$ and $G_2$ are cospectral but not isomorphic, whereas $G_1 - {x_1} $ and $G_2 -{x_2}$ are cospectral
because they are isomorphic.
Here $x_1$ and $x_2$ have degree 2.  The other vertex degrees are $(4,4,4,2,2,2)$
for $G_1$ and $(5,3,3,3,2,2)$ for $G_2$.  It is clear that the products of the degrees and the products of the degrees
minus 1 are different.  Thus, for any coalescences $G_1'$ and $G_2'$ of $(G_1, x_1)$ and $(G_2, x_2)$ with
any $(\Gamma,y)$, we have $Z_{G_1'} \ne Z_{G_2'}$ and $Z_{G_1'^*} \ne Z_{G_2'^*}$.
\end{ex}

We assumed that $G_1$ and $G_2$ are md2 or md1 in this proposition for simplicity, but this is not necessary.
What one really needs is a condition on pruned subgraphs of $G_1'$ and $G_2'$.

Note the coalescence construction includes the case of a disjoint union.
If $H_1$ and $H_2$ are cospectral, we can take $G_i$ to be
$H_i$ disjoint union a single vertex $x_i$.  Then $G_1$ and $G_2$ are cospectral, as are
 $G_1 - \{ x_1 \} = H_1$ and $G_2 - \{ x_2 \} = H_2$ by assumption.  Let $\Gamma$ be any graph and
 $y$ any vertex in $\Gamma$.  Then the coalescence of $(G_i, x_i)$ with $(\Gamma, y)$ is simply the disjoint
 union $H_i \sqcup \Gamma$ of $H_i$ with $\Gamma$.  However in the case of disjoint unions, we already know
  the stronger statement
  that $Z_{H_1} \ne Z_{H_2}$ implies $Z_{G_1'} \ne Z_{G_2'}$ since zeta functions factor into products over their
 connected components (though this factorization is not true for the zeta of the cones $Z_{G_i'^*}$).
 E.g., if $H_1$ and $H_2$ are the unique pair of cospectral graphs on 5 vertices, then
 $Z_{H_1 \sqcup \Gamma} \ne Z_{H_2 \sqcup \Gamma}$ for any $\Gamma$, though $G_1$ and $G_2$ are not md2.

\subsection{Join construction}

Suppose $G_1$ and $G_2$ are two graphs of order $n$ which have the same Laplacian spectra but different degree sequences.
Let $\Gamma$ be an arbitrary graph.  If $G_1$ or $G_2$ has an isolated vertex, then assume $\Gamma$ has at
least 2 vertices.  Then the joins
 $G_1 \vee \Gamma$ and $G_2 \vee \Gamma$ are connected and md2.

\begin{prop} With notation as above, $\Spec_L(G_1 \vee \Gamma) = \Spec_L(G_2 \vee \Gamma)$ for any $\Gamma$.
However $Z_{G_1 \vee \Gamma} \ne Z_{G_2 \vee \Gamma}$
for all but finitely many $\Gamma$.  Specifically, there is a finite set $S$ consisting of at most $n-1$ integers
 such that $Z_{G_1 \vee \Gamma} \ne Z_{G_2 \vee \Gamma}$ for any $\Gamma$ whose order $r$ does not lie in $S$.
\end{prop}

\begin{proof}
The first part is true for any $\Gamma$.

Let $d_1, \ldots, d_n$ be the degree sequence for $G_1$ and $d_1', \ldots, d_n'$ be the degree sequence for $G_2$.
Let $\delta_1, \ldots, \delta_r$ be the degree sequence for $\Gamma$.
Then the degree sequence for $G_1 \vee \Gamma$ is $d_1+r, \ldots, d_n + r, \delta_1 + n, \ldots, \delta_r + n$ and similarly
for $G_2 \vee \Gamma$.  Assume $r \ge 2$ so $G_1 \vee \Gamma$ and $G_2 \vee \Gamma$ are md2.
Consequently $Z_{G_1 \vee \Gamma}$ determines
\[ \prod (d_i + r - 1) \cdot \prod (\delta_j + n - 1) \]
and similarly for $G_2 \vee \Gamma$.  Hence $Z_{G_1 \vee \Gamma} \ne Z_{G_2 \vee \Gamma}$ if
\[  \prod (d_i + r - 1) \ne \prod (d_i' + r - 1). \]
Consider the polynomial
\[ f(x) = \prod (d_i + x) - \prod (d_i' + x), \]
which has degree $< n$.
The previous equation holds if and only if $r-1$ is not a root of $f(x)$, which can only happen for at most $n-1$ values of $r$.
\end{proof}

Note, one can replace the bound on the size of $S$ by the number of differing vertex degrees (counting multiplicity)
with $G_1$ and $G_2$.  Also, this proposition gives a non-algorithmic proof of Lemma \ref{lem:ds} by showing that any two order $n$ graphs with different degree sequences can by distinguished by the zeta functions of $n$ joins with graphs of distinct orders.

\begin{ex}
Let $G_1$ and $G_2$ be {\verb+ECZo+} and {\verb+EEr_+}.  These are graphs on 6 vertices with the same Laplacian spectra,
but $G_1$ has degree sequence $(4,2,2,2,2,2)$ and $G_2$ has degree sequence $(3,3,3,2,2,1)$.
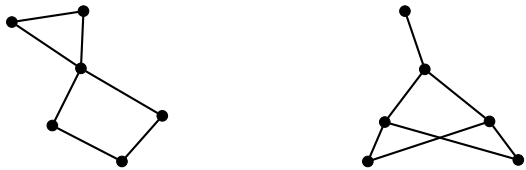
\begin{center}

\begin{tikzpicture}
\tikzset{VertexStyle/.style = {
                             shape = circle,
fill = black,
inner sep = 0pt,
outer sep = 0pt,
minimum size = 4pt,
draw} }
\useasboundingbox (0,0) rectangle (2.0cm,2.25cm);
\Vertex[L=\hbox{$.$},x=0.9523cm,y=2.0cm]{v0}
\Vertex[L=\hbox{$.$},x=2.0cm,y=0.6068cm]{v1}
\Vertex[L=\hbox{$.$},x=0.5385cm,y=0.4779cm]{v2}
\Vertex[L=\hbox{$.$},x=0.0cm,y=1.8548cm]{v3}
\Vertex[L=\hbox{$.$},x=1.4639cm,y=0.0cm]{v4}
\Vertex[L=\hbox{$.$},x=0.9185cm,y=1.2386cm]{v5}
\Edge[](v0)(v3)
\Edge[](v0)(v5)
\Edge[](v1)(v4)
\Edge[](v1)(v5)
\Edge[](v2)(v4)
\Edge[](v2)(v5)
\Edge[](v3)(v5)
\end{tikzpicture}
\hspace{1in}
\begin{tikzpicture}
\tikzset{VertexStyle/.style = {
                             shape = circle,
fill = black,
inner sep = 0pt,
outer sep = 0pt,
minimum size = 4pt,
draw} }
\useasboundingbox (0,0) rectangle (2.0cm,2.0cm);
\Vertex[L=\hbox{$.$},x=0.2249cm,y=0.5241cm]{v0}
\Vertex[L=\hbox{$.$},x=1.6176cm,y=0.5329cm]{v1}
\Vertex[L=\hbox{$.$},x=0.4933cm,y=2.0cm]{v2}
\Vertex[L=\hbox{$.$},x=0.0cm,y=0.0cm]{v3}
\Vertex[L=\hbox{$.$},x=2.0cm,y=0.0203cm]{v4}
\Vertex[L=\hbox{$.$},x=0.757cm,y=1.2255cm]{v5}
\Edge[](v0)(v3)
\Edge[](v0)(v4)
\Edge[](v0)(v5)
\Edge[](v1)(v3)
\Edge[](v1)(v4)
\Edge[](v1)(v5)
\Edge[](v2)(v5)
\end{tikzpicture}

\end{center}
For any (non-null) graph $\Gamma$,
the graphs $G_1 \vee \Gamma$ and $G_2 \vee \Gamma$ are automatically md2.
The polynomial $f(x)$ in the previous
proof is just $f(x) = (4+x)(2+x)^5 - (3+x)^3(2+x)^2(1+x) =(2+x)^2 (2x+5)$.  This has no positive roots, so
$Z_{G_1 \vee \Gamma} \ne Z_{G_2 \vee \Gamma}$ for any $\Gamma$.

The same conclusion is true if we replace $G_1$ and $G_2$ by their complements.
\end{ex}

\begin{cor} There are at least $2g_{n-6}$ pairs of graphs $(G_1, G_2)$ of order $n$ which have the same Laplacian spectra but
different Ihara zeta functions.
\end{cor}

\subsection{A new construction} \label{sec:new-con}

There is one pair of non-isomorphic md2 graphs on 9 vertices which have the same generalized characteristic polynomial $\phi_{ADJ}$, the pair from Example \ref{ex:9verts}.  This is the smallest such example.
This pair does not occur as the result of GM* switching for the following reason:
the matrices $A+xD$ are similar for all $x$, but there is no matrix $P$ such that $P^{-1}(A_1+xD_1)P = (A_2+xD_2)$ for all $x$, as must be the case for GM* pairs.
(In \cite{SS}, it was mistakenly written that all pairs of connected md2 graphs on
$n \le 11$ vertices with the same zeta function, adjacency spectrum, and Laplacian spectra are obtained by GM* switching, but this is false
as this example shows.  The second author of that paper informed us that the conclusion
of that sentence, ``are obtained by GM* switching,'' should be ``have the same $\phi_{ADJ}$.'')

There do however exist matrices $P$ and $R$ such that $(A_1+xD_1)(P+xR) = (P+xR)(A_2+xD_2)$ and $P+xR$ commutes with $J$ for all $x$ and is invertible for all real valued $x$.
Here we prove this example is part of a more general construction of pairs of graphs
with the same $\phi_{ADJ}$ which cannot be explained by GM* switching.  The fact that the
conjugating matrix $P+xR$ needs to depend on $x$ makes this construction more delicate (and
complicated) than GM* switching (and as far as we know, there are no other such constructions).
So we first describe it in a concrete way, and then remark afterwards
what are the necessary abstract conditions for this construction to work.

\begin{con} \label{our-construction}
Consider adjacency matrices $A_1, A_2$  of the following form:
$$ A_1 = \bmx B_1 & B_{12} & \cdots & B_{1k} & N_1 \\ {}^t B_{12} & B_2 & \cdots & B_{2k} & N_2 \\
\vdots & &\ddots & & \vdots \\ {}^t B_{1k} & \cdots & {}^t B_{k-1,k} &  B_k & N_k \\
{}^t  N_1 & \cdots & {}^t N_{k-1} & {}^t N_k & C \emx,
A_2 = \bmx B_1 &  {}^t B_{12} & \cdots &  {}^t B_{1k} & N_1 \\B_{12} & B_2 & \cdots &  {}^t B_{2k} & N_2 \\
\vdots & &\ddots & & \vdots \\ B_{1k} & \cdots &  B_{k-1,k} &  B_k & N_k \\
{}^t  N_1 & \cdots & {}^t N_{k-1} & {}^t N_k & C \emx$$
Here we allow the $B_i$'s to be chosen from among the following $4 \times 4$ adjacency matrices given in $2 \times 2$ block form by:
\[ \bmx 0 & 0\\  0 & 0 \emx, \bmx 0 & I \\ I & w \emx,\bmx 0 & J \\ J & 0 \emx, \bmx w & 0 \\ 0 & w \emx, \bmx w & J \\ J & w  \emx, \]
where $w = J - I$.
The $B_{ij}$'s ($i<j$) are allowed to be arbitrarily chosen from among the following 0-1 matrices with constant row and column sums:
\[  \bmx 0 & 0\\ 0 & 0 \emx, \bmx J & J \\ J & J  \emx, \bmx 0  & J \\ J & 0 \emx, \bmx J & 0 \\ 0 & J \emx, \bmx I & 0 \\ 0 & I\emx, \bmx w & 0\\ 0 & w \emx, \bmx 0 & I \\ w & 0 \emx, \bmx 0 & w \\ I & 0\emx, \]
\[ \bmx I & J \\ J & I\emx, \bmx w & J \\ J & w \emx, \bmx J & w \\ I & J \emx, \bmx J & I \\ w & J \emx, \bmx I & I \\ w & I \emx, \bmx I & w \\ I & I \emx, \bmx w & I \\ w & w \emx, \bmx w & w \\ I & w \emx.
\]
We allow $C$ to be any $m\times m$ adjacency matrix, and the $N_i$'s are $4\times m$ matrices such that each column of $N_i$ consists of all zeros or all ones.

Then $\phi_{ADJ}^{G_1} = \phi_{ADJ}^{G_2}$, where $G_i$ is the graph with adjacency matrix $A_i$.
\end{con}
\begin{proof}
Note that as each $B_{ij}$ has constant row and column sums and each $N_i$ has constant row sums, $G_1$ and $G_2$ have the same degree matrix, call this $D$.   Consider the following $4\times 4$ matrices:
$$Q = \bmx w & w-I\\ w-I & I \emx, S = \frac{1}{2}\bmx I-w & 0 \\ 0 & w-I\emx$$
Let $P$ be the block diagonal matrix whose first $k$ blocks are the matrix $Q$ and last block is the $m\times m$ identity matrix.  Let $R$ be the block diagonal matrix whose first $k$ blocks are the matrix $S$ and last block is the $m\times m$ zero matrix.

One can check that $(A_1+xD)(P+xR) = (P+xR)(A_2+xD)$ for all $x$.  Note that $P$ has constant row and column sums of 1 and $R$ has constant row and column sums of $0$ so $P+xR$ commutes with $J$ for all $x$.  Therefore $(A_1+xD + yJ)(P+xR) = (P+xR)(A_2+xD + yJ)$ for all $x,y$.  Also, $\det(P+xR) = \det(Q+xS)^k = (-x^2+2x-5)^k$ which is nonzero for all real $x$.  It follows that $A_1+xD + yJ$ and $A_2+xD + yJ$ have the same eigenvalues for all $x,y$ so $G_1$ and $G_2$ have the same $\phi_{ADJ}$.
\end{proof}

\begin{ex} \label{ex:9verts2} The pair of graphs $G_1$ and $G_2$ from Example \ref{ex:9verts} can be obtained from the above construction as follows.  Take $k=2$ and $B_1 = B_2 = \bmx 0 & I \\ I & w \emx$ and $B_{12}= \bmx I & I \\ w & I \emx$.  Take $C$ to be the $1\times 1$ zero matrix, $N_1$ to be the $4\times 1$ zero vector and $N_2$ to be the $4\times 1$ all ones vector.  The resulting graphs are the graphs from Example \ref{ex:9verts}, which we redraw below with the vertices labeled in the order that they appear in this description.
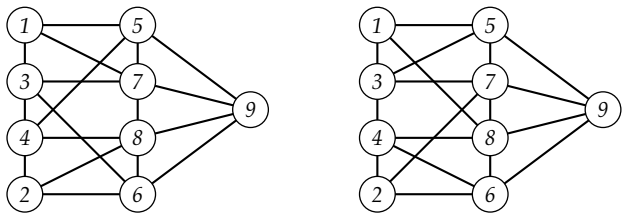
\begin{center}
\begin{tikzpicture}[scale=0.75,transform shape]
\Vertex[x=0,y=0]{2}
\Vertex[x=0,y=1]{4}
\Vertex[x=0,y=2]{3}
\Vertex[x=0,y=3]{1}
\Vertex[x=2,y=0]{6}
\Vertex[x=2,y=1]{8}
\Vertex[x=2,y=2]{7}
\Vertex[x=2,y=3]{5}
\Vertex[x=4,y=1.5]{9}
 \Edge(1)(3)
 \Edge(3)(4)
 \Edge(4)(2)
 \Edge(5)(7)
 \Edge(7)(8)
 \Edge(8)(6)
 \Edge(1)(5)
 \Edge(1)(7)
 \Edge(3)(7)
 \Edge(3)(6)
 \Edge(4)(5)
 \Edge(4)(8)
 \Edge(2)(8)
 \Edge(2)(6)
 \Edge(5)(9)
 \Edge(6)(9)
 \Edge(7)(9)
 \Edge(8)(9)
\end{tikzpicture}
\hspace{1cm}
\begin{tikzpicture}[scale=0.75,transform shape]
\Vertex[x=0,y=0]{2}
\Vertex[x=0,y=1]{4}
\Vertex[x=0,y=2]{3}
\Vertex[x=0,y=3]{1}
\Vertex[x=2,y=0]{6}
\Vertex[x=2,y=1]{8}
\Vertex[x=2,y=2]{7}
\Vertex[x=2,y=3]{5}
\Vertex[x=4,y=1.5]{9}
 \Edge(1)(3)
 \Edge(3)(4)
 \Edge(4)(2)
 \Edge(5)(7)
 \Edge(7)(8)
 \Edge(8)(6)
 \Edge(1)(5)
 \Edge(1)(8)
 \Edge(3)(7)
 \Edge(3)(5)
 \Edge(4)(6)
 \Edge(4)(8)
 \Edge(2)(7)
 \Edge(2)(6)
 \Edge(5)(9)
 \Edge(6)(9)
 \Edge(7)(9)
 \Edge(8)(9)
\end{tikzpicture}
\end{center}

\end{ex}

If in the previous example, we take $B_1, B_2, B_{12}$ as before but let $C$ be any adjacency matrix and $N_1 = 0$ and $N_2=J$ (appropriately sized all zero and all ones matrices respectively), then this construction always results in a non-isomorphic pair of graphs.  Hence this construction produces at least $g_{n-8}$ non-isomorphic graphs with the same $\phi_{ADJ}$ on $n$ vertices.
On 10 vertices, there are 6 different non-isomorphic pairs of graphs that can be built using this construction.

This construction can be made more general.  Given a set of $n\times n$ adjacency matrices $B_i$ and a set of $n\times n$ matrices $B_{ij}$ with 0 and 1 entries and constant row and column sums, then we can do an analogous construction if there exist $n\times n$ matrices $Q$ and $S$ with the following properties.  First, $Q$ and $S$ should have constant row and column sums of 1 and 0 respectively and have $\det(Q+xS)$ not equal to the zero polynomial.  If $D_i$ is the degree matrix of $M_i$, then we need $QB_i= B_iQ$, $D_iS=SD_i$, and $D_iQ-QD_i = SB_i-B_iS$ for all $i$.  Also, we need $B_{ij}Q = Q{}^tB_{ij}$, ${}^tB_{ij} Q = QB_{ij}$, $B_{ij}S = S{}^tB_{ij}$, and ${}^tB_{ij}S = SB_{ij}$.  If all of these properties are satisfied, then any $A_1$ and $A_2$ built in the same manner as in the construction will result in graphs with the same $\phi_{ADJ}$.  Note that at least one of the $B_{ij}$ must be non-symmetric for $A_1$ and $A_2$ to be different, but this does not guarantee that the resulting graphs will be non-isomorphic.

\section{Calculations and Conjectures} \label{sec5}

In this section, we present our computational study of comparing graphs via the usual spectral methods ($A$, $L$ and $|L|$)
with methods (M1)--(M4).  Our findings, together with the above results, lead us to several formal and informal conjectures.

In \cite{HS}, Haemers and Spence enumerated the number of graphs on $n \le 11$ vertices which are not DS, not $L$-DS, and not $|L|$-DS (see also \cite{BS2} and \cite{spence:web} for some errata
and additional data for $n=12$).  This data is summarized in the $A$, $L$ and $|L|$ columns
of Table \ref{tab:main}.  They also enumerated the number of non-DS graphs explained by 2-GM switching and 2-GM* switching.
This data is recalled in the 2-GM and 2-GM* columns of Table \ref{tab:main}.  The 2-GM column is a lower bound for the $A$ column,
and the 2-GM* will be a lower bound for all other columns in the table.

Our main numerical results are a similar enumeration of the number of graphs on $n \le 11$ vertices
which are not determined by methods (M1)--(M4) (considering these methods separately---considering them in tandem is equivalent to just using (M4)).
The $\barth$ column of Table \ref{tab:main} gives the number of order $n$ graphs which are not \DZb, i.e., not determined by (M3).
(Recall that non-\DZb\ implies non-DS, non-$L$-DS, and non-$|L|$-DS.)
It---surprisingly to us---turned out that for $n \le 11$ being DZ*, \DZZ\ and \DZZb\ (i.e.,
not determined by (M1), (M2) or (M4)) are all equivalent.  The number of graphs
which do not satisfy these conditions is given in
the $\barth \bar{\barth}$ column.

\begin{table}[!h]
\small
\caption{Counting graphs not determined by spectral invariants \newline
(M3) is in the $\barth$ column;
(M1), (M2) and (M4) are in the $\barth \bar \barth$ column}
\centering
\begin{tabular}{r|r|r|r|r|r|r|r|r}
$n$  & \# graphs & $A$ & $L$ & $|L|$ & $\barth$ & $\barth \bar \barth$ & 2-GM & 2-GM*\\
\hline
$\le 5$ & 51 & 2 & 0 & 6 & 0 & 0 &0&0\\
6 & 156 & 10 & 4 & 16&0 & 0&0&0\\
7 & 1,044   & 110 & 130 & 102& 0& 0&40 & 0\\
8 & 12,346   & 1,722 & 1,767 & 1,201 &0 &0 &1,054 & 0\\
9 & 274,668   & 51,038  & 42,595 & 19,001 & 2& 2 & 38,258 & 0\\
10 & 12,005,168  & 2,560,606 & 1,412,438 & 636,607 & 10,146 & 10,140 &2,047,008 & 9,480\\
11 & 1,018,997,864  &  215,331,676 & 91,274,836  & 38,966,935 &  1,353,426 & 1,353,402  & 176,895,408 & 1,297,220
\end{tabular}
\label{tab:main}
\end{table}

This data suggests several things: (i) there is not much difference among (M1)--(M4); (ii) any of (M1)--(M4)
seems to do much better than $A$, $L$ or $|L|$; (iii) most GM-pairs are distinguished by (M1)--(M4);
(iv) most graphs not distinguished by (M1)--(M4) are explained by 2-GM* switching.
(We note that for $n=11$, 216 additional non-\DZZb\ graphs are explained by 3-GM* switching, but none for
$n \le 10$.\footnote{To enumerate graphs obtained by 3-GM* switching, we observe that to get
a non-isomorphic switch $\~G$ from $G$, one needs at least one of the $B_i$'s to
be even of size $\ge 4$, say $B_1$.  One may also assume $B_2$ has size $> 1$,
and the size of $B_2$ is not relatively prime to the size of $B_1$.  Then
for each $G$ which is not distinguished by $\phi_{ADJ}$, one can check if it satisfies the
3-GM* condition by iterating first through all ``admissible'' choices for $V_{B_1}$,
then through choices for $V_{B_2}$.}  Several more non-\DZZb\ graphs
for $9 \le n \le 11$ may also be explained by Construction \ref{our-construction}.)
These suggestions are
congruous with our work in previous sections.  We just state one of them
as a formal conjecture.

\begin{conj} \label{conj:main} We have: (i)
almost all graphs are DZ*; and (ii) almost all non-DS graphs are DZ*.  Further, (iii) if $\mathcal H_n$
(resp.\ $\mathcal H'_n$)
is the set of non-DS (resp.\ non-DZ*) graphs of order $n$, then $\# \mathcal H'_n / \# \mathcal H_n \to 0$
as $n \to \infty$.

The above statements are also true if we replace DZ* with any of \DZZ, \DZb, and \DZZb.
\end{conj}

In light of Haemers' conjecture that almost all graphs are DS,  (ii)
seems to be much stronger than (i).  Note (iii) is strictly stronger than (ii) as not all DS graphs are DZ*.
One reason to believe (ii) is that it appears a significant fraction
of non-DS graphs are explained by GM switching.  Theorem \ref{thm:GM-zeta} suggests that most of these
are distinguished by (M1)--(M4).  Further, most of the graphs not distinguished by (M1)--(M4) are explained
by GM* switching, which appears to account for just a trivial fraction of graphs obtained by GM switching
(see \cite{HS}).
Note this reasoning does not apply to distinguishing graphs which are not $L$-DS or $|L|$-DS.

In fact, Theorem \ref{thm:GM-zeta} suggests that just the Ihara zeta function $Z_G$ by itself might
differentiate most pairs of
graphs with the same adjacency spectrum.  However, if we are just looking at $Z_G$, it is more reasonable
to restrict to md2 graphs.  Let us say an md2 graph is DZ if there is no other md2 graph of the same order with
the same Ihara zeta function.

Setyadi--Storm \cite{SS} enumerated all pairs of connected md2 graphs with the same Ihara zeta function
for $n \le 11$.  This does not exactly tell us the number of non-DZ md2 graphs (even among
connected graphs), so we
did a similar enumeration to Table \ref{tab:main} for md2 graphs with $n \le 10$.
The results are in Table \ref{tab:md2}, with the column headings meaning the same things as in
Table \ref{tab:main}, except restricted to md2 graphs.  The only new column is the $Z$ column, which
is the number of non-DZ md2 graphs on $n$ vertices.

\begin{table}[!h]

\small
\caption{\label{tab:md2} Counting md2 graphs not determined by spectral invariants}
\centering
\begin{tabular}{r|r|r|r|r|r|r|r}
$n$  & \# graphs & $A$ & $L$ & $|L|$ & $Z$ & $\barth$ &  $\barth \bar \barth$ \\
\hline
$\le 6$ & 77  & 0 & 0 & 4 & 0 & 0 & 0 \\
7 & 510  & 26 & 64 & 37 &0  & 0 & 0 \\
8 & 7,459 & 744 &1,156 & 725 & 2 & 0 & 0 \\
9 & 197,867  & 32,713 & 31,353 & 13,878  & 6 & 2 & 2 \\
10 & 9,808,968 & 1,727,629 & 1,184,460 & 535,080 & 10,130 & 10,094 & 10,088
\end{tabular}
\end{table}

We remark that while computing this data, we discovered some small errors in the tables in \cite{SS} for $n=11, 12$.  Namely, in the cases where more than 2 connected md2 graphs have the same Ihara zeta function,
\cite{SS} undercounts the number of pairs.  For instance, for $n=10$, the entries in the \cite{SS}
tables should be augmented by 1 for $m=20, 21, 24, 25$.  In each of these cases, there is one triple of
connected md2 graphs all with the same zeta function.

This data, together with our work above, suggests that, when restricting to md2 graphs, there is little difference between
methods (M1)--(M4) and just using the Ihara zeta function.  Hence we are led to the following conjecture.

\begin{conj} \label{conj:md2}
Almost all md2 graphs are DZ, and almost all non-DS md2 graphs are DZ.
In fact, if $\mathcal H_n^{(2)}$
(resp.\ $\mathcal H_n^{\prime (2)}$)
is the set of non-DS (resp.\ non-DZ) md2 graphs of order $n$, then $\# \mathcal H_n^{\prime (2)} / \# \mathcal H_n^{(2)} \to 0$
as $n \to \infty$.
\end{conj}

Since almost all graphs are md2 for $n$ large, this first statement would imply that almost all graphs of order $n$
are determined by their Ihara zeta function.

\medskip
We performed the above calculations by generating the graphs with fixed $n$ and $m$ using nauty \cite{nauty} in Sage 6.1.1
\cite{sage}.   Then we used Sage to compute $\det|L| = \lim_{t \to -1} (1-t^2)^{n-m} Z_G(-1)^{-1}$ and
$\det(4D+2A-3I) = (-3)^{n-m} Z_G(-2)^{-1}$ for each $G$, and similarly for $G^*$ and $\bar G$.  By sorting the graphs
according to these values for $G^*$ or $G$ and $\bar G$, we made a list of candidate non-DZ* and non-\DZZ\ graphs.
For each of these pairs, we check to see if the corresponding Ihara zeta functions match up.
We compute the Ihara zeta functions by first constructing $G^\pr$, then using the Hashimoto determinant formula.  This
is much faster than using Sage 6.1.1's built-in function to compute $Z_G(t)^{-1}$.  To conserve memory, for $n = 10, 11$,
we wrote out the data of the special values of the zeta function with the graph6 string to a file for fixed $n$ and $m$,
and sorted these files using the Unix {\verb+sort+} tool.  Then for each pair of graphs with the same $Z_G$ and $Z_{\bar G}$,
we compared their generalized characteristic polynomials $\phi_{ADJ}$.  This takes care of (M1), (M2) and (M4).
We will return to (M3) momentarily.

This method of making a first pass by checking just 2 values of the zeta function $Z_G(-1)$ and $Z_G(-2)$ is based
on the heuristic idea that it is unlikely that 2 graphs $G_1$ and $G_2$ will have the same spectra for several
different random linear combinations of $A$ and $D$, unless they have the same spectra for all linear combinations
of $A$ and $D$.  (This is part of the reason we believe the Ihara zeta function has essentially the same distinguishing
power as the Bartholdi zeta function.)  From some computational experimentation, we discovered that
just knowing $n$, $m$ and the 2 values $Z_G(-1)$ and $Z_G(-2)$ almost always determines $Z_G(t)$ (for $n \le 11$).

This heuristic also suggests that using 2 independent spectra to distinguish graphs is much better than a single one.
To further test this idea, we enumerated graphs which are not determined by the following sets of spectra:
$A$; $A$ and $L$; $A$, $L$ and $|L|$; all three of these plus the Ihara zeta function, and finally $\phi_{AD}$.
These numbers are respectively given in the 3rd through 7th columns of Table \ref{tab:combine}.

\begin{table}[!h]

\small
\caption{\label{tab:combine} Counting graphs not determined by combining spectral invariants}
\centering
\begin{tabular}{r|r|r|r|r|r|r}
$n$  & \# graphs & $A$ & $AL$ & $AL|L|$ & $AL|L|Z$ & $\barth$ \\
\hline
$\le 5$ & 51 & 2 & 0 & 0 & 0 & 0 \\
6 & 156 & 10 & 0 & 0 &0 & 0\\
7 & 1,044   & 110 & 0  & 0 & 0& 0\\
8 & 12,346   & 1,722 & 0 & 0  &0 &0\\
9 & 274,668   & 51,038  & 82 & 2 & 2& 2 \\
10 & 12,005,168  & 2,560,606 & 13,948&  10,718 & 10,150 & 10,146 \\
11 & 1,018,997,864  &  215,331,676 & 1,468,790  & 1,361,246 & 1,353,498 & 1,353,426
\end{tabular}
\end{table}

Note that even though the numbers of graphs not distinguished by $A$ or $L$ or $|L|$ individually are
quite large (cf.\ Table \ref{tab:main}), the number of graphs not distinguished by combining 2 or more of these
shrinks drastically.  (We did not compute the numbers for combining $A$ and $|L|$ or $L$ and $|L|$ for all
$n \le 11$, but we expect a similar phenomenon to hold in these cases also.  For instance when $n=9$,
there are only 8 graphs not distinguished by $A$ and $|L|$, and 4,405 graphs not distinguished
by $L$ and $|L|$.)
This suggests that, say, using $A$ and either $L$ or $|L|$, to distinguish graphs is closer in
effectiveness to using (M1)--(M4) than to using just $A$ or just $L$ or just $|L|$.

The calculations for Table \ref{tab:combine} were obtained by comparing spectra via spectral moments,
and successively sieving out graphs.  E.g., once we know all pairs or graphs with the same $A$- and $L$-spectra,
we search through these to see which also have the same $|L|$-spectra.  To compute the last column (i.e., (M3)),
we looked at all pairs with the same $A$-, $L$-, $|L|$-spectra and Ihara zeta functions, and checked their
generalized characteristic polynomials $\phi_{AD}$.

\medskip
We conclude with a couple of final remarks about related calculations.

First, Brouwer and Spence \cite{BS2} find that quite large families can have the same adjacency
spectrum (46 graphs can have the same spectrum on 11 vertices, and this goes up to 128 on 12 vertices).
It seems much rarer for a large family of graphs to have the same zeta functions.  We found that
on 10 vertices there are 4 triples of graphs with the same $\phi_{ADJ}$, and no larger families.
On 11 vertices, there are
1,442 triples of graphs with the same $\phi_{ADJ}$, and 192 quadruples, but no larger families.

Second, one might wonder about using cones or complements with the usual spectra.  Knowing
the Laplacian spectrum of $G$ is equivalent to knowing it for $\bar G$ or $G^*$, but what about
$A$ or $|L|$?  From Corollary \ref{cor:cone-comp}, we know that knowing the $A$- or $|L|$-spectrum of $G$ and $\bar G$ implies the same for $G^*$ (and conversely knowing the $A$- or $|L|$-spectrum of $G^*$ and $\overline{G^*}$ implies the same for $G$).  In Table \ref{tab:cone-comp},
we enumerate the graphs $G$ on $n \le 10$ vertices which are not determined by
the following: spectrum of $G^*$, spectra of $G$ and $\bar G$, $|L|$-spectrum of $G^*$
and $|L|$-spectrum of $G$ and $\bar G$, listed respectively in columns 3 through 6.
The data in the $A \bar A$ columns is already in \cite{HS}.
We note there has been recent work towards showing almost all graphs are distinguished by
the $A$- and $\bar A$-spectra---see, e.g., \cite{wang} or \cite{vDH2}.

\begin{table}[!h]
\small
\caption{\label{tab:cone-comp} Counting graphs not determined by cones and complements for $A$ and $|L|$}
\centering
\begin{tabular}{r|r|r|r|r|r}
$n$  & \# graphs & $A^*$ & $A \bar A$ & $|L|^*$ & $ |L| \overline{|L|}$  \\
\hline
$\le 4$ & 17 & 0 & 0 & 2 & 2  \\
5 & 34 & 0 & 0 & 4 & 4  \\
6 & 156 & 0 & 0 & 16 & 16 \\
7 & 1,044  & 44 & 40 & 102 & 102\\
8 & 12,346  & 1,194 & 1,166 & 1,139 & 1,139 \\
9 & 274,668  & 44,120 & 43,811 & 18,748 & 18,748  \\
10 & 12,005,168 & 2,423,121 & 2,418,152 & 633,232 & 633,226 \\
\end{tabular}

\end{table}

We remark that for $A$- or $|L|$-spectra, there appears to be little difference between using
$G$ and $\bar G$ versus $G^*$ versus just $G$ (cf.\ Table \ref{tab:main}; see \cite{HS}, \cite{spence:web}
for the $A \bar A$ data when
$n=11, 12$).  However, unlike the case of $L$-spectra, there is at least some difference.
In particular, the $A$- or $|L|$-spectrum of $G^*$ does not always determine that
of $G$, but Table \ref{tab:cone-comp} suggests it usually does (particularly for $|L|$).

The coincidence of methods (M1) and (M2) on $n \le 11$ vertices means that for
graphs with $n \le 11$ vertices of size $m$, $Z_{G^*}$ always determines $Z_G$ and $Z_{\bar G}$,
and vice versa.
Consequently, for $n \le 11$ vertices, knowing $Z_{G^*}$ is equivalent to knowing
$Z_{(\bar G)^*}$.
By analogy with the $A$-, $L$-, and $|L|$-spectra,
Corollary \ref{cor:cone-comp} suggests that $Z_G$ and $Z_{\bar G}$ (at least almost) always
determine $Z_{G^*}$, and Table \ref{tab:cone-comp} suggests that $Z_{G^*}$ almost always
determines $Z_G$ and $Z_{\bar G}$.
That is, we expect (M1) and (M2) to be almost always equivalent, but there
seems to be no reason to expect this is always the case.  Similarly, we have no
reason to expect (M2) and (M4) always give the same results.


\begin{bibdiv}

\begin{biblist}

\bib{bartholdi}{article}{
   author={Bartholdi, Laurent},
   title={Counting paths in graphs},
   journal={Enseign. Math. (2)},
   volume={45},
   date={1999},
   number={1-2},
   pages={83--131},
   issn={0013-8584}
}

\bib{bass}{article}{
   author={Bass, Hyman},
   title={The Ihara-Selberg zeta function of a tree lattice},
   journal={Internat. J. Math.},
   volume={3},
   date={1992},
   number={6},
   pages={717--797},
   issn={0129-167X}
}


\bib{BS}{article}{
   author={Bayati, Paymun},
   author={Somodi, Marius},
   title={On the Ihara zeta function of cones over regular graphs},
   journal={Graphs Combin.},
   volume={29},
   date={2013},
   number={6},
   pages={1633--1646},
   issn={0911-0119}
}

\bib{BH}{book}{
   author={Brouwer, Andries E.},
   author={Haemers, Willem H.},
   title={Spectra of graphs},
   series={Universitext},
   publisher={Springer, New York},
   date={2012},
   pages={xiv+250},
   isbn={978-1-4614-1938-9}
}

\bib{BS2}{article}{
   author={Brouwer, A. E.},
   author={Spence, E.},
   title={Cospectral graphs on 12 vertices},
   journal={Electron. J. Combin.},
   volume={16},
   date={2009},
   number={1},
   pages={Note 20, 3 pp},
   issn={1077-8926}
}

\bib{buser}{book}{
    AUTHOR = {Buser, Peter},
     TITLE = {Geometry and spectra of compact {R}iemann surfaces},
    SERIES = {Modern Birkh\"auser Classics},
      NOTE = {Reprint of the 1992 edition},
 PUBLISHER = {Birkh\"auser Boston, Inc., Boston, MA},
      YEAR = {2010},
     PAGES = {xvi+454},
      ISBN = {978-0-8176-4991-3},
       URL = {http://dx.doi.org/10.1007/978-0-8176-4992-0}
}

\bib{CRS}{book}{
   author={Cvetkovi{\'c}, Drago{\v{s}}},
   author={Rowlinson, Peter},
   author={Simi{\'c}, Slobodan},
   title={An introduction to the theory of graph spectra},
   series={London Mathematical Society Student Texts},
   volume={75},
   publisher={Cambridge University Press, Cambridge},
   date={2010},
   pages={xii+364},
   isbn={978-0-521-13408-8}
}

\bib{cooper}{article}{
   author={Cooper, Yaim},
   title={Properties determined by the Ihara zeta function of a graph},
   journal={Electron. J. Combin.},
   volume={16},
   date={2009},
   number={1},
   pages={Research Paper 84, 14 pp},
   issn={1077-8926}
}

\bib{czarneski}{book}{
   author={Czarneski, Debra},
   title={Zeta functions of finite graphs},
   note={PhD Thesis, Louisiana State University and Agricultural \&
   Mechanical College},
   publisher={ProQuest LLC, Ann Arbor, MI},
   date={2005},
   pages={61},
   isbn={978-0542-25438-3}
}

	

\bib{GM}{article}{
   author={Godsil, C. D.},
   author={McKay, B. D.},
   title={Constructing cospectral graphs},
   journal={Aequationes Math.},
   volume={25},
   date={1982},
   number={2-3},
   pages={257--268},
   issn={0001-9054}
}

\bib{HS}{article}{
   author={Haemers, Willem H.},
   author={Spence, Edward},
   title={Enumeration of cospectral graphs},
   journal={European J. Combin.},
   volume={25},
   date={2004},
   number={2},
   pages={199--211},
   issn={0195-6698}
}

\bib{hashimoto}{article}{
   author={Hashimoto, Ki-ichiro},
   title={Zeta functions of finite graphs and representations of $p$-adic
   groups},
   conference={
      title={Automorphic forms and geometry of arithmetic varieties},
   },
   book={
      series={Adv. Stud. Pure Math.},
      volume={15},
      publisher={Academic Press, Boston, MA},
   },
   date={1989},
   pages={211--280}
}

\bib{hashimoto2}{article}{
   author={Hashimoto, Ki-ichiro},
   title={On zeta and $L$-functions of finite graphs},
   journal={Internat. J. Math.},
   volume={1},
   date={1990},
   number={4},
   pages={381--396},
   issn={0129-167X}
}

\bib{hashimoto3}{article}{
   author={Hashimoto, Ki-ichiro},
   title={Artin type $L$-functions and the density theorem for prime cycles
   on finite graphs},
   journal={Internat. J. Math.},
   volume={3},
   date={1992},
   number={6},
   pages={809--826},
   issn={0129-167X}
}

\bib{ihara}{article}{
   author={Ihara, Yasutaka},
   title={On discrete subgroups of the two by two projective linear group
   over ${\germ p}$-adic fields},
   journal={J. Math. Soc. Japan},
   volume={18},
   date={1966},
   pages={219--235},
   issn={0025-5645}
}

\bib{JN}{article}{
   author={Johnson, Charles R.},
   author={Newman, Morris},
   title={A note on cospectral graphs},
   journal={J. Combin. Theory Ser. B},
   volume={28},
   date={1980},
   number={1},
   pages={96--103},
   issn={0095-8956}
}

\bib{KL}{article}{
   author={Kim, Hye Kyung},
   author={Lee, Jaeun},
   title={A generalized characteristic polynomial of a graph having a
   semifree action},
   journal={Discrete Math.},
   volume={308},
   date={2008},
   number={4},
   pages={555--564},
   issn={0012-365X}
}

\bib{nauty}{article}{
   author={McKay, Brendan D.},
   author={Piperno, Adolfo},
   label = {nauty},
   title={Practical graph isomorphism, II},
   journal={J. Symbolic Comput.},
   volume={60},
   date={2014},
   pages={94--112},
   issn={0747-7171}
}


\bib{sage}{misc}{
  label          = {Sage},
  Author       = {W.\thinspace{}A. Stein and others},
  Organization = {The Sage Development Team},
  Title        = {{S}age {M}athematics {S}oftware ({V}ersion 6.1.1)},
  note         = {{\tt http://www.sagemath.org}},
  Year         = {2014}
}

\bib{schwenk}{article}{
   author={Schwenk, Allen J.},
   title={Almost all trees are cospectral},
   conference={
      title={New directions in the theory of graphs},
      address={Proc. Third Ann Arbor Conf., Univ. Michigan, Ann Arbor,
      MI},
      date={1971},
   },
   book={
      publisher={Academic Press, New York},
   },
   date={1973},
   pages={275--307}
}

\bib{scott-storm}{article}{
   author={Scott, Geoffrey},
   author={Storm, Christopher},
   title={The coefficients of the Ihara zeta function},
   journal={Involve},
   volume={1},
   date={2008},
   number={2},
   pages={217--233},
}

\bib{SS}{article}{
   author={Setyadi, A.},
   author={Storm, C. K.},
   title={Enumeration of graphs with the same Ihara zeta function},
   journal={Linear Algebra Appl.},
   volume={438},
   date={2013},
   number={1},
   pages={564--572},
   issn={0024-3795}
}

\bib{spence:web}{article}{
   author = {Spence, E.},
   title = {Personal homepage},
   note = {\url{http://www.maths.gla.ac.uk/~es}, accessed July 15, 2014}
}

\bib{ST}{article}{
   author={Stark, H. M.},
   author={Terras, A. A.},
   title={Zeta functions of finite graphs and coverings},
   journal={Adv. Math.},
   volume={121},
   date={1996},
   number={1},
   pages={124--165},
   issn={0001-8708}
}

\bib{terras}{book}{
   author={Terras, Audrey},
   title={Zeta functions of graphs: A stroll through the garden},
   series={Cambridge Studies in Advanced Mathematics},
   volume={128},
   publisher={Cambridge University Press, Cambridge},
   date={2011},
   pages={xii+239},
   isbn={978-0-521-11367-0}
}

\bib{vDH1}{article}{
   author={van Dam, Edwin R.},
   author={Haemers, Willem H.},
   title={Which graphs are determined by their spectrum?},
   note={Special issue on the Combinatorial Matrix Theory Conference
   (Pohang, 2002)},
   journal={Linear Algebra Appl.},
   volume={373},
   date={2003},
   pages={241--272},
   issn={0024-3795}
}

\bib{vDH2}{article}{
   author={van Dam, Edwin R.},
   author={Haemers, Willem H.},
   title={Developments on spectral characterizations of graphs},
   journal={Discrete Math.},
   volume={309},
   date={2009},
   number={3},
   pages={576--586},
   issn={0012-365X}
}

\bib{wang}{article}{
   author={Wang, Wei},
   title={Generalized spectral characterization of graphs revisited},
   journal={Electron. J. Combin.},
   volume={20},
   date={2013},
   number={4},
   pages={Paper 4, 13 pp},
   issn={1077-8926}
}

\bib{WLLX}{article}{
   author={Wang, Wei},
   author={Li, Feng},
   author={Lu, Hongliang},
   author={Xu, Zongben},
   title={Graphs determined by their generalized characteristic polynomials},
   journal={Linear Algebra Appl.},
   volume={434},
   date={2011},
   number={5},
   pages={1378--1387},
   issn={0024-3795}
}

\end{biblist}

\end{bibdiv}
\end{document}